\documentclass[aos,preprint]{imsart}

\RequirePackage[OT1]{fontenc}
\usepackage{amsfonts,amssymb,amsmath,amsthm,mathrsfs,bm}
\RequirePackage{natbib}
\RequirePackage[colorlinks,citecolor=blue,urlcolor=blue]{hyperref}

\usepackage{graphicx}
\usepackage[margin=1cm]{caption}
\usepackage{framed}
\usepackage{enumerate}
\usepackage{subfigure}
\usepackage{soul}
\usepackage{float}
\usepackage{algorithmic}
\usepackage{algorithm}
\usepackage[usenames,dvipsnames]{color}
\definecolor{shadecolor}{gray}{0.85}
\newcommand{\bs}{\boldsymbol}

\arxiv{1710.01295}

\startlocaldefs

\def\quadruples'{\mathop{\sum\!\sum\!\sum\!\sum}}
\newtheorem{thm}{Theorem}
\newtheorem{corollary}{Corollary}
\newtheorem{prop}{Proposition}
\newtheorem{defn}{Definition}

\newtheorem{lemma}{Lemma}

\endlocaldefs

\begin{document}

\begin{frontmatter}
\title{Isotropic covariance functions on graphs and their edges\thanksref{T1}}
\runtitle{Covariance functions on graphs and their edges}
\thankstext{T1}{We thank Heidi S\o gaard Christensen, James Sharpnack, Adrian Baddeley, and Gopalan Nair for useful comments and illuminating discussions. We are very grateful to an anonymous referee who pointed out that the work of \cite{cambanis1983charfuns}, \cite{berens19971} and \cite{gneiting1998charfuns} can be used to prove Theorems~\ref{thm: extension for tree graphs} and \ref{thm: simplified extension for tree graphs}.}

\begin{aug}
\author{\fnms{Ethan Anderes}\thanksref{t1}\ead[label=e1]{anderes@ucdavis.edu}},
\author{\fnms{Jesper M{\o}ller}\thanksref{t3}\ead[label=e2]{jm@math.aau.dk}},
\and
\author{\fnms{Jakob G.\ Rasmussen}\thanksref{t3}\thanksref{t4}\ead[label=e3]{jgr@math.aau.dk}}

\thankstext{t1}{
Research supported in part by NSF grants DMS-1252795, DMS-1812199 and a UC Davis Chancellor's Fellowship.
}
\thankstext{t3}{Supported by
 The Danish Council for Independent Research | Natural Sciences, grant DFF – 7014-00074
"Statistics for point processes in space and beyond", and by the "Centre for Stochastic Geometry and
Advanced Bioimaging", funded by grant 8721 from the Villum Foundation.}
\thankstext{t4}{Supported by the Australian Research Council, Discovery Grant DP130102322.}
\runauthor{E. ANDERES, J. M{\O}ller AND J. G. RASMUSSEN}

\affiliation{University of California at Davis 
and Aalborg University
}

\address{E. Anderes\\
Department of Statistics\\
University of California at Davis\\
One Shields Avenue, Davis CA 95616\\
USA\\
\printead{e1}} 

\address{J. M{\o}ller\\
J. G. Rasmussen\\
Department of Mathematical Sciences\\
Aalborg University\\
Skjernvej 4A, DK-9220 Aalborg {\O}\\
Denmark\\
\printead{e2}
\\
\phantom{E-mail:\ }\printead*{e3}}

\end{aug}

\begin{abstract}
We develop parametric classes of covariance functions on linear networks and their extension to graphs with Euclidean edges, i.e., graphs with edges viewed as line segments or more general sets with a coordinate system allowing us to consider points on the graph which are vertices or points on an edge. Our covariance functions are defined on the vertices and edge points of these graphs and are isotropic in the sense that they depend only on the geodesic distance or on a new metric called the resistance metric
(which extends the classical resistance metric developed in electrical network theory on the vertices of a graph to the continuum of edge points). We discuss the advantages of using the resistance metric in comparison with the geodesic metric as well as the restrictions these metrics impose on the investigated covariance functions.  In particular, many of the commonly used isotropic covariance functions  in the spatial statistics literature (the power exponential, Mat{\'e}rn, generalized Cauchy, and Dagum classes)   are shown to be valid with respect to the resistance metric for any graph with Euclidean edges, whilst they are only valid with respect to the geodesic metric in more special cases.
\end{abstract}



\end{frontmatter}

%
%

\section{Introduction}\label{s:intro}

Linear networks are used to model a wide variety of non-Euclidean spaces
occurring in applied statistical problems involving river networks, road networks, and dendrite networks, see e.g.\
\cite{cressie:frey:harch:smith:06}, \cite{cressie1997spatio}, \cite{gardner2003predicting}, \cite{hoef:peterson:theobald:06}, \cite{hoef:peterson:10}, \cite{okabe:sugihara:12}, and \cite{baddeley:rubak:turner:15}.
However, the problem of developing valid random field models over networks is a decidedly difficult task. 
Compared to what is known for Euclidean spaces -- where the results of Bochner and Schoenberg characterize the class of all stationary covariance functions, see e.g.\ \cite{yaglom1987book} -- the corresponding results for linear networks are few and far between.
Even the fundamental notion of a stationary covariance function is, at best, ambiguous for linear networks. However, the notion of an isotropic covariance function can be made precise by requiring the function to depend only on a metric defined over the linear network.
Often the easiest choice for such a metric is given by the length of the shortest path connecting two points, i.e., the geodesic metric. Still there are no general results which establish when a given function generates a valid isotropic covariance function with respect to this metric. Indeed, \cite{baddeley-et-al:17} concluded that spatial point process models on a linear network with a pair correlation function which is only depending on shortest path distance \lq\lq may be quite rare".

In this paper, we use Hilbert space embedding techniques to establish that many of the flexible isotropic covariance models used in spatial statistics are valid 
over linear networks with respect the geodesic metric and a new metric introduced in Section~\ref{d_R def}. This new metric is called the resistance metric because it extends the classical resistance metric developed in electrical network theory.
The validity of these covariance models do not hold, however, over the full parametric range available in Euclidean spaces. Moreover, we show the results for the geodesic metric apply to a much smaller class of linear networks and can not be extended to a graph that has three or more paths connecting two points on the linear network. This is in stark contrast to the resistance metric where we show there is no restriction on the type of linear network for which they apply.

We develop a generalization of a linear network  which we call a {graph with Euclidean edges}. Essentially, this is a graph $(\mathcal V, \mathcal E)$ where each edge $e\in \mathcal E$ is additionally associated to an abstract set in bijective correspondence with a line segment of $\mathbb R$. Treating the edges as abstract sets allows us to consider points on the graph that are either vertices or points on the edges, and the bijective assumption gives each edge set a (one-dimensional) Cartesian coordinate system for measuring distances between any two points on the edge (therefore the terminology Euclidean edges). The within-edge Cartesian coordinate system will be used to extend the geodesic and the resistance metric on the vertex set to the whole graph (including points on the edges). 
Our objective then is to construct parametric families of covariance functions over graphs with Euclidean edges which are isotropic with respect to the geodesic metric and the resistance metric developed below (in fact our covariance functions  will be (strictly) positive definite). 
Thereby a rich class of isotropic Gaussian random fields on the whole graph can be constructed and inferred via likelihood methods. Finally, we remark that the validity of these isotropic covariance functions also allows the construction of isotropic point process models on the whole graph constructed via a log Gaussian Cox process  \citep{moeller:syversveen:waagepetersen:98,moeller:waagepetersen:03}. We leave this and other applications of our paper for future work.


\subsection{Graphs with Euclidean edges}\label{s:DefGraphWithEuclEdges}

A linear network is typically defined as the union of a finite collection of line
segments in $\mathbb R^2$ 
with distance between two points defined as the length of the shortest path
connecting the points.  This definition, although conceptually clear, does have limitations that restrict their application. For example, in the case of road networks:
\begin{itemize}
\item Bridges and tunnels can generate networks which do not have a planar
representation as a union of line segments in $\mathbb R^2$.
\item Varying speed
limits or number of traffic lanes may require distances on line segments to be
measured differently than their spatial extent.
\end{itemize}
A graph with Euclidean edges, defined below, is a generalization of linear networks that easily
overcomes the above-mentioned limitations while still retaining the salient
feature relevant to applications: that edges (or line segments) have a Cartesian
coordinate system associated with them.

\begin{defn}\label{def.gwee}
    A triple $\mathcal G=(\mathcal V, \mathcal E, \{\varphi_e\}_{e\in\mathcal E})$ which satisfies the following conditions (a)-(d) is called a
  {\rm graph with Euclidean edges}.
    \begin{enumerate}[{\rm (a)}]
        \item\label{def.gwee1}
        {\rm Graph structure}: 
        $(\mathcal V, \mathcal E)$ is a finite simple connected graph, meaning that the vertex set $\mathcal V$ is finite,
        the graph has no repeated edges or edge which joins a vertex to itself, and every pair of vertices is connected by a path.
        \item\label{def.gwee2}
        {\rm Edge sets}: 
        Each edge $e\in \mathcal E$ is associated with a unique abstract set, also denoted $e$, where the vertex set $\mathcal V$ and all the edge sets $e\in \mathcal E$ are mutually disjoint.
        \item\label{def.gwee3}
        {\rm Edge coordinates}: 
        For each edge $e\in \mathcal E$, if $u,v \in \mathcal V$ are the vertices connected by $e$, then $\varphi_e$ is a bijection defined on $e\cup\{u, v\}$ (the union of the edge set $e$ and the vertices $\{u, v\}$) such that
        $\varphi_e$
        maps $e$ onto an open interval $(\underline e, \overline e)\subset \mathbb R$ and $\{u, v\}$ onto the endpoints $\{\underline e, \overline e\}$. 
        \item\label{def.gwee4}
        {\rm Distance consistency}: 
        Let
        $d_{\mathcal G}(u,v)\colon \mathcal V\times \mathcal V\to  [0,\infty)$ denote the standard shortest-path weighted graph metric on the vertices of $(\mathcal V, \mathcal E)$ with edge weights given by $\overline e - \underline e$ for every $e\in \mathcal E$. Then, for each $e\in \mathcal E$ connecting two vertices $u,v \in \mathcal V$, the following equality holds:
            \[ d_{\mathcal G}(u,v) = \overline e - \underline e.
            \]
    \end{enumerate}
  We write $u \in \mathcal G$ as a synonym for $u\in\mathcal V \cup \bigcup_{e\in \mathcal E}e$, the whole graph given by the union of $\mathcal V$ and all edges $e\in\mathcal E$.
\end{defn}

If we consider a linear
network $\cup_{i\in \mathcal I} \ell_i$ consisting of closed line segments $\ell_i\subset \mathbb R^2$ which intersect only at their endpoints, we can easily construct a graph with Euclidean edges as follows.
Let $\mathcal V$ be the set of endpoints of the line segments. Let each edge set $e_i\in\mathcal E$ correspond to the relative interior of the corresponding line segment $\ell_i$. Let each bijection $\varphi_{e_i}$ be given by the inverse of the path-length parameterization of $\ell_i$. Then conditions (a)-(d) are easily seen to hold.

\begin{figure} 
    \centering
    \includegraphics[height=1.5in]{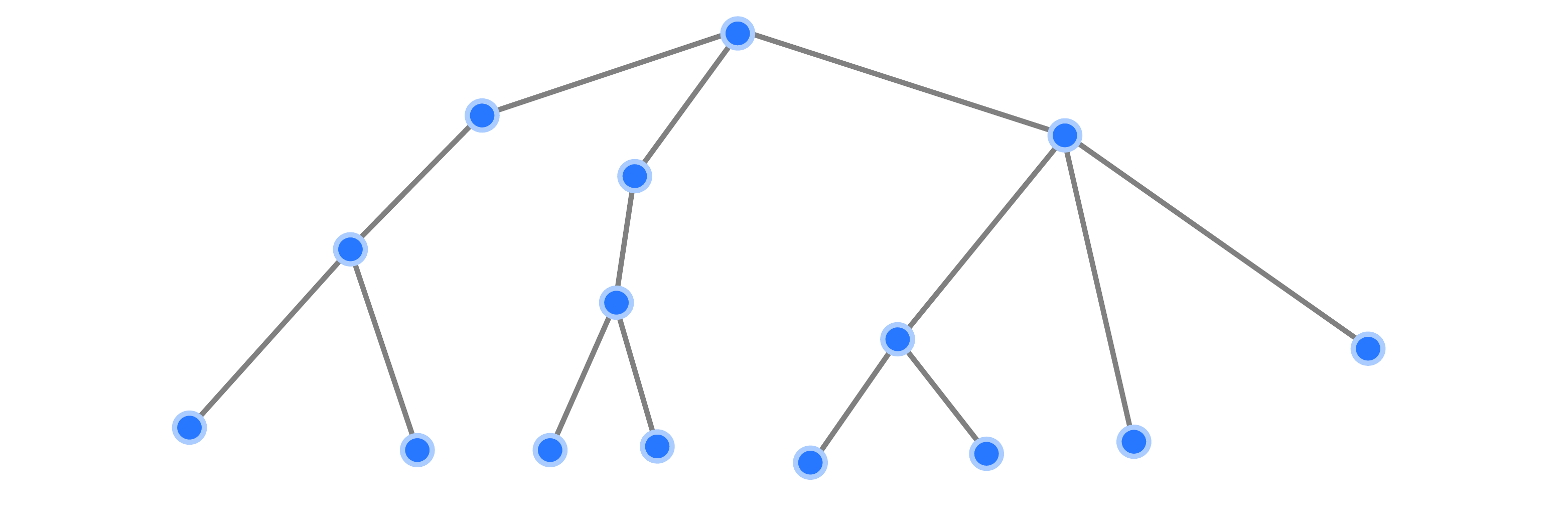}
    \caption{A Euclidean tree constructed from the linear network of grey lines.
    The blue dots represent the vertices. 
    }
    \label{tree gwee}
\end{figure}

Any triple $\mathcal G = (\mathcal V, \mathcal E, \{\varphi_e\}_{e\in \mathcal E})$ for which $(\mathcal V, \mathcal E)$ forms a tree graph is automatically a graph with Euclidean edges given that conditions \eqref{def.gwee2} and \eqref{def.gwee3} hold. In this case, $\mathcal G$ is said to be a {\it Euclidean tree}. Figure~\ref{tree gwee} shows an example.
If the graph $(\mathcal V, \mathcal E)$, associated with a graph with Euclidean edges $\mathcal G$, forms a cycle, then $\mathcal G$ is said to be a {\it Euclidean cycle}.
Conversely, if $(\mathcal V, \mathcal E)$ forms a cycle graph with edge bijections $\{\varphi_e\}_{e\in \mathcal E}$, then the resulting triple $\mathcal G:=(\mathcal V, \mathcal E, \{\varphi_e\}_{e\in \mathcal E})$ satisfies the conditions of Definition~\ref{def.gwee}  whenever there are three or more vertices (to ensure there are no multiple edges) and for every $e_o\in \mathcal E$ the following inequality is satisfied:
\begin{equation}
    \overline e_o - \underline e_o\leq \sum_{\substack{e\in\mathcal E \\ e\neq e_o}}
     (\overline e - \underline e).
\end{equation}
The above condition guarantees that no edge spans more than \textit{half} of the circumference of the cycle, implying that distance consistency holds for $\mathcal G$.
 Figure~\ref{two cycles isometrically equiv} illustrates examples of Euclidean cycles (the two first graphs) and an example of a graph violating both conditions (\ref{def.gwee1}) and (\ref{def.gwee4}) in Definition \ref{def.gwee} (the last graph).

\begin{figure} 
    \centering
    \includegraphics[height=1.5in]{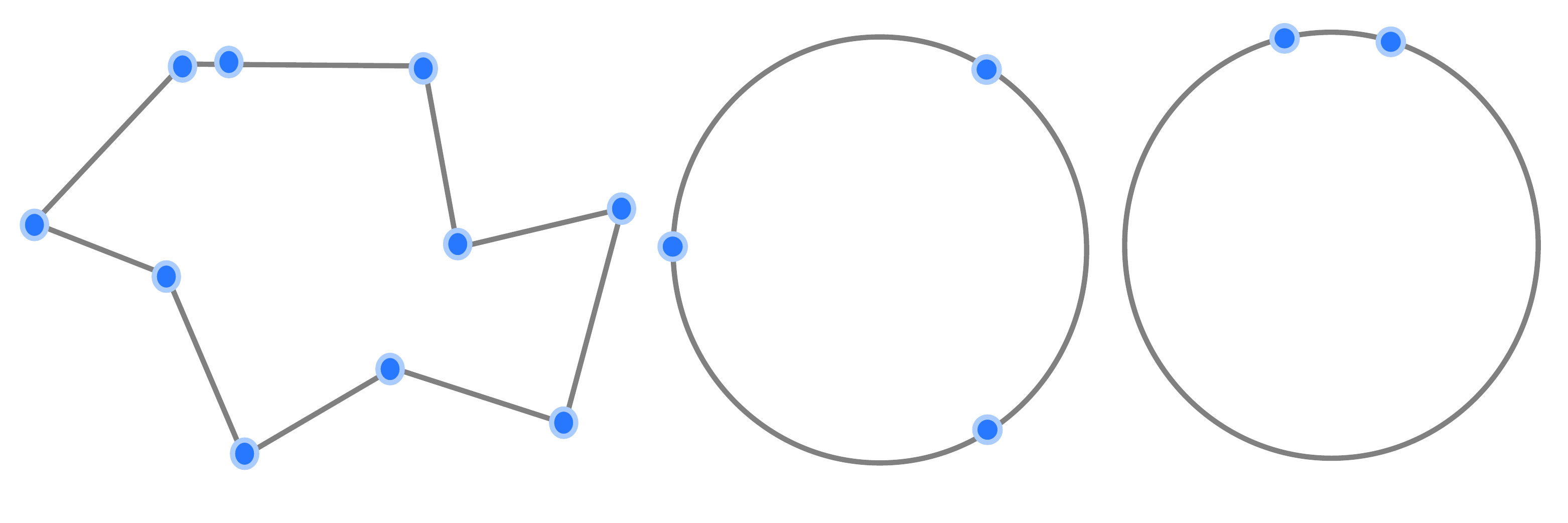}
    \caption{
	The two graphs on the left are Euclidean cycles. However, the right most graph is {\it not} 
    a graph with Euclidean edges.
	}
    \label{two cycles isometrically equiv}
\end{figure}

\begin{figure}[H]
    \centering
    \includegraphics[height=1.5in]{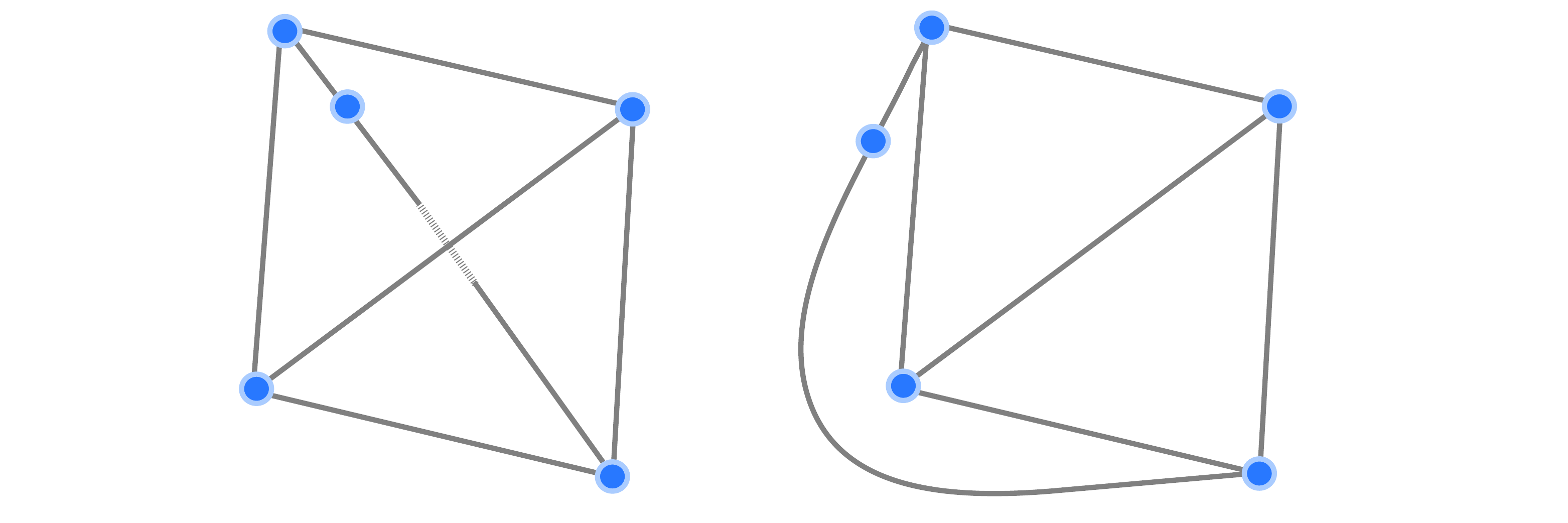}
    \caption{
	The two diagrams above show a graph with Euclidean edges $\mathcal G$ which can {\it not} be represented as a linear network in $\mathbb R^2$. The diagram on the left is drawn in a way that visually preserves edge length but forces an intersection that does not correspond to a vertex in $\mathcal G$ (the dashed segment indicates that one edge passes under the other). The diagram on the right is drawn without non-vertex intersections but requires curved segments that have length which do not correspond to the lengths determined by the edge bijections for a linear network.
	}
    \label{not a linear network}
\end{figure}

In all the above examples we have used spatial curves and line segments to represent the edges. It it worth pointing out that this is simply a visualization device. Indeed, the structure of a graph with Euclidean edges is completely invariant to the geometric shape of the visualized edges just so long as the path-length of each edge is preserved.
This concept is important when considering the example given in Figure~\ref{not a linear network}, where the edge represented by the diagonal line in the leftmost drawing represents a bridge or tunnel bypassing the other diagonal edge. Hence the lack of vertices at the intersection with that edge. Note that it is impossible to avoid this intersection when lengths of edges are fixed. This implies that this graph with Euclidean edges can not be represented as a linear network in $\mathbb R^2$.


\subsection{Summary of main results}\label{ea: summary of main results}

This section presents our main theorems
explicitly, leaving the proofs and precise definitions for later 
sections.

Our first contribution is to establish sufficient conditions for a
function $C:[0,\infty)\to \mathbb R$ to generate a  (strictly) positive definite
function of the form  $C(d(u,v))$ where $d(u,v)$ is a metric defined over the vertices and edge points of a graph with Euclidean edges $\mathcal G$; then we call $\mathcal G\times\mathcal G\ni (u,v)\to C(d(u,v))\in \mathbb R$ an {\it isotropic covariance function} and $C$ its {\it radial profile}. We study two metrics, the {\it geodesic metric}, $d_{G,\mathcal G}$, as defined in Section~\ref{d_G def}, and a new {\it resistance metric}, $d_{R,\mathcal G}$, as developed in Section \ref{d_R def},  which extends the resistance metric on the vertex set -- from electrical network theory \citep{KleinRandic:93} -- to the continuum of edge points on $\mathcal G$. As is apparent from the following two theorems, there are fundamental differences in terms of the generality of valid isotropic covariance functions when measuring distances under the two metrics.

\begin{figure}
    \centering
    \includegraphics[height=1.5in]{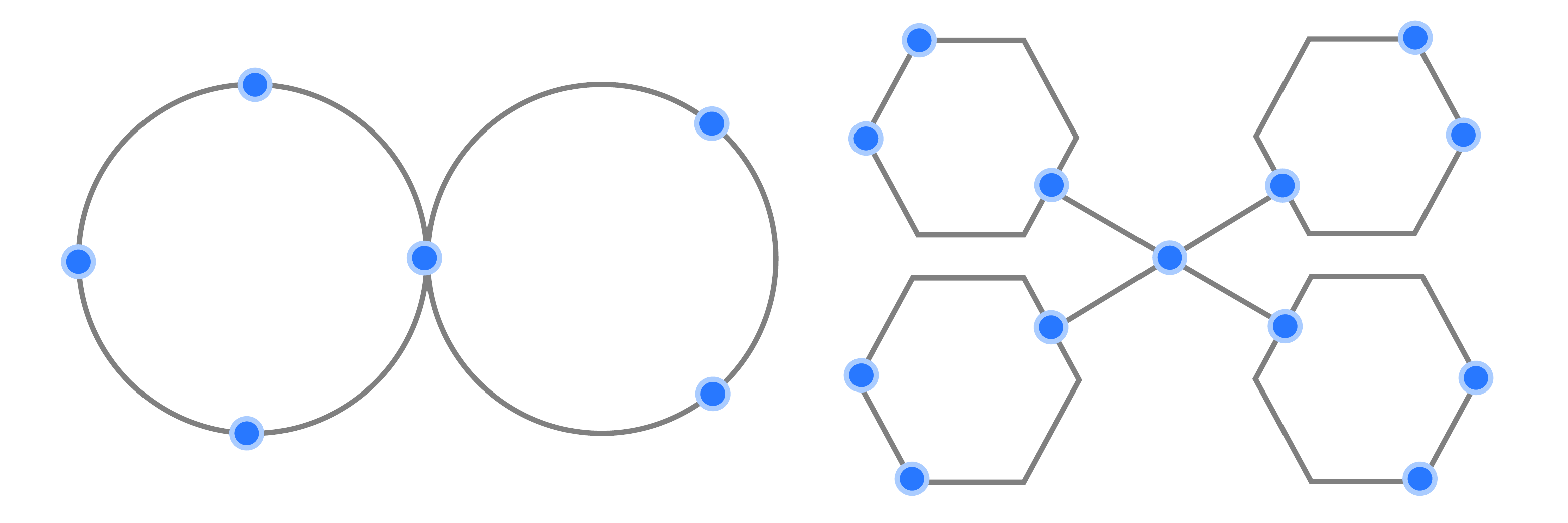}
        \caption{Examples of finite sequential 1-sums of cycles and trees. Left: A 1-sum of two Euclidean cycles. Right: A sequential 1-sum of four Euclidean cycles and one Euclidean tree.}
        \label{ea: one sum examples}
\end{figure}

In Theorem~\ref{main thm: iso cov funs} below, we consider the \textit{1-sum} of two graphs with Euclidean edges $\mathcal G_1$ and $\mathcal G_2$ having only a single point in common, $\mathcal G_1\cap\mathcal G_2=\{x_0\}$. This is defined explicitly in Section~\ref{ea: Hilbert space embedding section}, but the concept is easy to visualize as the merging of $\mathcal G_1$ and $\mathcal G_2$ at $x_0$ and the concept easily extends to the case of three or more graphs with Euclidean edges; Figure~\ref{ea: one sum examples} gives two graphical illustrations. Further, we need to recall the following definition of a completely monotonic function, noting there is a distinction, in the literature,  between  complete monotonicity on $[0,\infty)$ versus on $(0,\infty)$, the latter being fundamentally related to Bernstein functions and variograms (see \cite{berg2008stieltjes, wells1975embeddings}). 

\begin{defn}
    A function $f:[0,\infty) \to \mathbb R$ is said to be
    {\it completely monotonic on $[0,\infty)$} 
    if $f$ is continuous on $[0,\infty)$, infinitely differentiable on $(0,\infty)$ and \mbox{${(-1)}^{j}f^{(j)}(t)\geq 0$} over $ (0,\infty)$ for every integer $j\geq 0$,  where $f^{(j)}$ denotes the $j$th derivative of $f$ and $f^{(0)}=f$.
\end{defn}

\begin{thm}\label{main thm: iso cov funs}
    Let $C : [0, \infty) \rightarrow \mathbb R$ be a completely monotone and non-constant function.
    \begin{enumerate}[{\rm (i)}]
     \item\label{main thm: iso cov funs, item 1} If $\mathcal G$ is a graph with Euclidean edges then $C(d_{R,\mathcal G}(u,v))$ is (strictly) positive definite over $(u, v) \in \mathcal G \times \mathcal G$.
    \item\label{main thm: iso cov funs, item 2} If $\mathcal G$ is a graph with Euclidean edges that forms a finite sequential 1-sum of Euclidean cycles and trees then $C(d_{G,\mathcal G}(u,v))$ is (strictly) positive definite over $(u,v) \in \mathcal G \times\mathcal G$.
    \end{enumerate}
\end{thm}

\begin{table}
     \centering
 \begin{tabular}{lll}
     \textit{Type} & \textit{Parametric form} & \textit{Parameter range}
     \\\hline\\[-0.2cm]
     {Power exponential} & $C(t)=\exp(-\beta t^\alpha)$ & $0 < \alpha \leq 1$,  $\beta  > 0$. \\[0.2cm]

     {Mat\'ern}  &
     $C(t)=\frac{2^{1-\alpha}}{\Gamma(\alpha)}(\beta t)^\alpha K_{\alpha}(\beta t)$ &
     $0 < \alpha \leq \frac{1}{2}$, $\beta >0$. \\[0.2cm]

      {Generalized Cauchy}  & $C(t)=(\beta t^\alpha + 1)^{-\xi/\alpha}$ & $0 < \alpha \leq 1$, $\beta,\xi > 0$. \\[0.2cm]

     {Dagum }&
     $C(t)=\left[{\displaystyle 1-\left(\frac{\beta t^\alpha}{1+ \beta t^\alpha } \right)^{\xi/\alpha}}\right]$ & $0 < \alpha\leq 1$, $0 < \xi \leq 1$, $\beta >0$. \\[0.5cm]

     \hline
 \end{tabular}
 \caption{Parametric classes of functions $C:[0,\infty)\to\mathbb R$ which generate 
 isotropic correlation functions $C(d_{R,\mathcal G}(\cdot,\cdot))$, i.e., when distance is measured by the resistance metric and $C(0)=1$. Note: $K_{\alpha}$ denotes the modified Bessel function of the second kind and order $\alpha$. }
 \label{ea:table of acf}
 \end{table}

A consequence of Theorem \ref{main thm: iso cov funs} is that many of the parametric classes of auto-covariance functions used in spatial statistics are (strictly) positive definite with respect to $d_{R,\mathcal G}$ for general graphs and with respect to $d_{G,\mathcal G}$ for graphs which are $1$-sums of Euclidean trees and cycles. Notice, however, this holds only after restricting the parametric range to ensure complete monotonicity on $[0,\infty)$, as in the radial functions given in Table \ref{ea:table of acf}.  To see why the radial functions in Table \ref{ea:table of acf} are completely monotonic, first note that $t\to f(\beta t^\alpha + \lambda)$ is completely monotonic if $\alpha\in(0,1]$, $\beta,\lambda > 0 $, and $f$ is completely monotonic  (see Equation (1.6) in \cite{miller2001completely}). Therefore, $\exp(-\beta t^\alpha)$ and $(\beta t^\alpha +1)^{-\xi/\alpha}$ are completely monotonic for $\beta,\xi > 0$ and $\alpha\in(0,1]$, since both $\exp(-t)$ and $(t+1)^{-\xi/\alpha}$ are completely monotonic. This establishes the desired result for the {power exponential} class and the {generalized Cauchy} class in Table \ref{ea:table of acf}. The complete monotonicity for the {Mat\'ern} class in Table \ref{ea:table of acf} was proved in Example 2 of \cite{gneiting2013strictly}.
Finally, Theorem 9 in \cite{berg2008dagum} establishes that $C(t)=1-\big(t^\beta/(1+t^\beta)\big)^\gamma$ is completely monotonic whenever $\beta\gamma \in (0,1]$ and $\beta \in (0,1]$, which proves the desired result for the {Dagum} class in Table \ref{ea:table of acf}.

In the special case where $\mathcal G$ is a Euclidean tree with geodesic metric $d_{G,\mathcal G}(u,v)$, the results of Theorem \ref{main thm: iso cov funs}\eqref{main thm: iso cov funs, item 2} can be obtained from existing literature. Indeed, it is well known that the exponential covariance functions are positive definite (via $\ell_1$ embedding, using Theorem~4.1 in \cite{wells1975embeddings} and Theorem~3.2.2 in \cite{deza1996geometry}), which implies that positive mixtures of exponential
covariance functions are positive definite with respect to $d_{G,\mathcal G}(u,v)$.  
Now the results of Schoenberg (outlined in Theorem \ref{thm:charcomplmonofuns2} below) are sufficient to establish Theorem~\ref{main thm: iso cov funs}\eqref{main thm: iso cov funs, item 2} for this special case.

The contrasting generality of the range of graphs $\mathcal G$ applicable in Theorem~\ref{main thm: iso cov funs} for \eqref{main thm: iso cov funs, item 1} versus \eqref{main thm: iso cov funs, item 2}  hints at a degeneracy that occurs when modeling covariance functions which are isotropic with respect to the geodesic metric. The next result, Theorem~\ref{ea:forbidden subparagraph}, confirms this degeneracy by showing that for the geodesic metric, Theorem~\ref{main thm: iso cov funs} can not be extended to the generality given for the resistance metric. Here, for $S=R$ or $S=G$, if there exists some $\beta>0$ so that $e^{-\beta\, d_{S,\mathcal G}(u,v)}$ is not a positive semi-definite function over $(u,v)\in\mathcal G\times \mathcal G$, we say that $\mathcal G$ is a {\it forbidden graph (for the exponential class) with respect to the metric $d_{S,\mathcal G}$}. Figure~\ref{forbidden subgraph plot} shows examples in case of the geodesic metric. Note that if a forbidden graph is present as a subgraph of $\mathcal G$, then $\mathcal G$  is forbidden as well.

\begin{figure}
    \centering
    \includegraphics[height=1.5in]{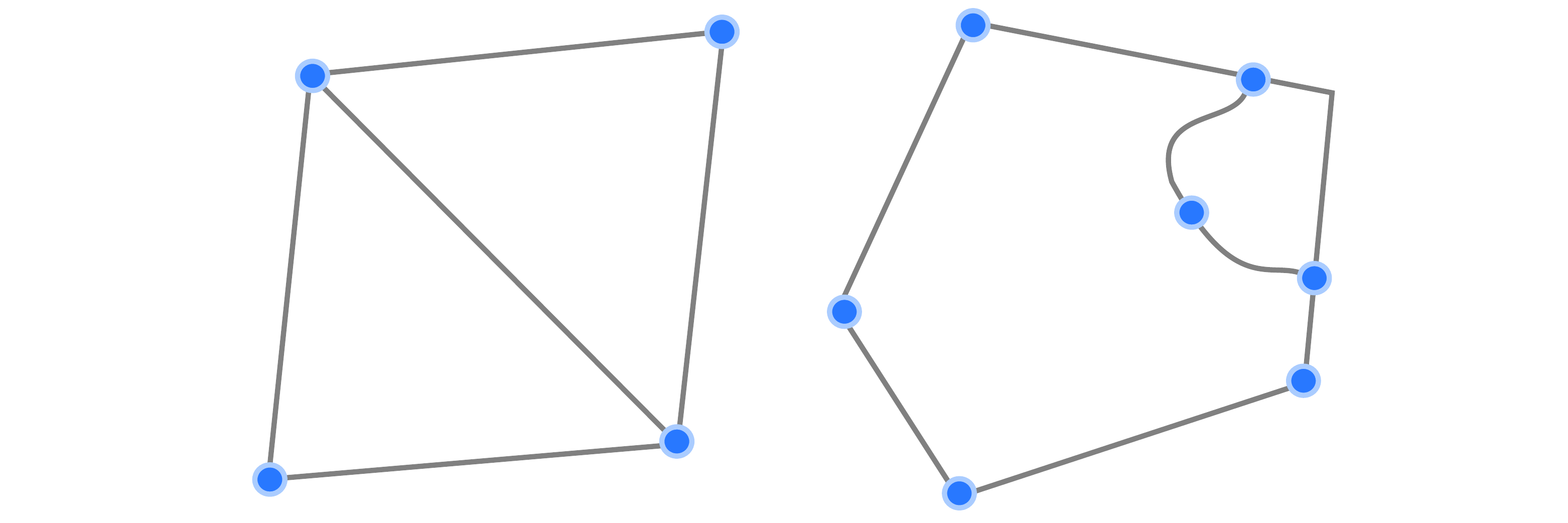}%
    \caption{Examples of forbidden graphs for the exponential class with respect to the geodesic metric.
    }\label{forbidden subgraph plot}
\end{figure}

\begin{thm}\label{ea:forbidden subparagraph}
    If  $\mathcal G$ is a graph with Euclidean edges for which there exists
    three distinct paths connecting two points $u,v\in\mathcal G$, then $\mathcal G$ is a forbidden graph for the exponential class with respect to the geodesic metric.
\end{thm}

In Section \ref{d_R def}, we develop the new resistance metric $d_{R,\mathcal G}$ which is specifically designed to overcome the restrictions imposed by Theorem~\ref{ea:forbidden subparagraph} for the geodesic metric on general graphs with Euclidean edges. We construct $d_{R,\mathcal G}$ as the variogram of a canonical Gaussian random field over $\mathcal G$ obtained by linearly interpolating a random vector on the vertices constructed from the graph Laplacian then adding independent Brownian bridges over each edge. While it is known that the (discrete) effective resistance metric can be expressed as the variogram of a random vector (see \cite{Lyons:2017:PTN:3086816} for an excellent exposition) it appears that the approach given here -- namely, using the variogram of a canonical (continuous) random field to define a (continuum) resistance metric -- is new. The advantage of this construction is that it gives the following key Hilbert space embedding result.

\begin{thm}\label{ea: main embedding for dR}
    If  $\mathcal G$ is a graph with Euclidean edges, there exists a Hilbert space $H$ and an embedding $\varphi \colon \mathcal G\to H$ such that
    \begin{equation}\label{ea: display for main embedding}
    \sqrt{d_{R,\mathcal G}(u,v)} = \big\|\varphi(u) - \varphi(v) \big\|_H
    \end{equation}
    for all $u,v\in \mathcal G$ where $d_{R,\mathcal G}$ is the resistance metric developed in Section \ref{d_R def}.
    If, in addition, $\mathcal G$ forms a sequential 1-sum of a finite number of Euclidean cycles and trees, then the above result also holds for the geodesic metric $d_{G,\mathcal G}$.
\end{thm}

In some sense, the construction of $d_{R,\mathcal G}$ in Section \ref{d_R def} and the proof of Theorem~\ref{ea: main embedding for dR} are the most important results of this paper. Once they are established, many of the results in this section follow almost immediately from well known consequences of Schoenberg's work in the context of embeddings, see e.g.\  \cite{wells1975embeddings} or \cite{jayasumana2013kernel}.

The next two theorems illustrate the scope of the results given above. In particular, Theorem \ref{main thm: iso cov funs}\eqref{main thm: iso cov funs, item 1} gives sufficient conditions for (strict) positive definiteness over {\em all} graphs with Euclidean edges $\mathcal G$. This does not preclude a less stringent sufficient condition that holds for a sub-collection of graphs with Euclidean edges. For example, consider the case where $\mathcal G$ has a single edge connecting two vertices. Then both $d_{R,\mathcal G}$ and $d_{G,\mathcal G}$ are equivalent to the Euclidean metric on a compact interval $[0,c]\subset \mathbb R$ and, as such, contains a much richer collection of positive definite covariance functions than those established in Theorem~\ref{main thm: iso cov funs}. A less trivial example can be obtained by restricting $\mathcal G$ to be a Euclidean tree as in the next two theorems.

\begin{thm}\label{thm: extension for tree graphs}
Let $\mathcal G$ be a Euclidean tree with $m$ leaves, where $m\geq 3$. Then $C(d_{G,\mathcal G}(u,v))$ and  $C(d_{R,\mathcal G}(u,v))$ are positive semi-definite over $(u,v)\in\mathcal G\times \mathcal G$ whenever $C\colon[0,\infty)\to \mathbb R$ is given by
\begin{equation}\label{cambanis char}
C(t) = \int_{0}^\infty \omega_{\lceil m/2 \rceil}(\sigma t) \,\mathrm d\mu(\sigma)
\end{equation}
where $\mu$ is a finite (positive) measure on $(0,\infty)$ and $\omega_n(t)$ is defined by
\begin{align*}
\omega_n(t) &= \frac{\Gamma(n/2)}{\sqrt{\pi}\,\Gamma((n-1)/2)}\int_1^\infty \Omega_n(v^{1/2}t)\,v^{-n/2}(v-1)^{(n-3)/2}\, \mathrm dv
\end{align*}
with $\Omega_n(t) = \Gamma(n/2)(2/t)^{(n-2)/2} J_{(n-2)/2}(t)$ and $J_\nu(t)$ denoting the Bessel function of the first kind and order $\nu$. 
\end{thm}

The proof of the above result, given in Section~\ref{forbidden subgraph for d_G}, follows directly by the work of \cite{cambanis1983charfuns}  once it is established that Euclidean trees with $m$ leaves can be embedded in $\mathbb R^{\lceil m/2 \rceil}$ with $\ell_1$ metric and $d_{R,\mathcal G} = d_{G,\mathcal G}$ on Euclidean trees (cf.\ Proposition~\ref{relation btwn dG and dr} below). This $\ell_1$ embedding result can also used in combination with Theorem 3.2 of \cite{gneiting1998charfuns} to give a simplified criterion for the conclusion of Theorem~\ref{thm: extension for tree graphs}.

\begin{thm}\label{thm: simplified extension for tree graphs}
Let $\mathcal G$ be a Euclidean tree with $m$ leaves, where $m\geq 3$. If $C\colon[0,\infty)\to \mathbb R$ is a continuous function such that $C^{(2\lceil m/2 \rceil-2)}$ is convex and $\lim_{t\rightarrow \infty}C(t) = 0$,  then both $C(d_{G,\mathcal G}(u,v))$ and  $C(d_{R,\mathcal G}(u,v))$ are positive semi-definite over $(u,v)\in\mathcal G\times \mathcal G$. 
\end{thm}

Finally, we notice that Theorem~\ref{thm: extension for tree graphs} shows that covariance functions on Euclidean trees may attain negative values, and at the very end of Section~\ref{sec.forbidden} we give an example of a parametric family of covariance functions whose support can be made arbitrary small.


\subsection{Outline for the remainder of the paper}

Details of the geodesic metric and resistance metric over graphs with Euclidean edges, along with their theoretical properties used in subsequent sections, are given in Section~\ref{ea: metric development}. The resistance metric $d_{R,\mathcal G}$ is defined constructively as the variogram of a certain random field over $\mathcal G$, analogous to a Wiener process on $\mathbb R$. This construction has the advantage that it establishes the Hilbert space embedding result almost immediately (utilizing a theorem of Schoenberg). The difficulty, however, is in showing that $d_{R,\mathcal G}$ is indeed an extension of the classical effective resistance on any finite subgraph and is invariant to the graph operations of splitting and merging edges (cf.\ Proposition~\ref{splitting and merging edges} below). The invariance result is important since it implies the resistance metric is, in some sense, intrinsic to the minimal graph structure of $\mathcal G$: the addition or removal of unnecessary vertices along an edge leaves $d_{R,\mathcal G}$ unchanged.  Proofs of these theoretical properties rely heavily on Hilbert space methods and are deferred to the Appendix.

Sections \ref{ea: Hilbert space embedding section} and \ref{forbidden subgraph for d_G} contain the proofs of all the results summarized in Section~\ref{ea: summary of main results} and follow relatively easily given the results in Section~\ref{ea: metric development}. The Hilbert space embedding stated in Theorem~\ref{ea: main embedding for dR} is proved first in Section~\ref{ea: Hilbert space embedding section}. The remaining four theorems summarized in Section~\ref{ea: summary of main results} are proved in Section~\ref{forbidden subgraph for d_G}. Finally, in Section \ref{sec.forbidden} we establish constraints -- depending on the graph structure of $\mathcal G$ -- for features of any radial profile which generates an isotropic covariance function with respect to either metric, resistance or geodesic.

%
%

\section{The geodesic and resistance metric on $\mathcal G$}\label{ea: metric development}

In this section we develop the geodesic and resistance metric over graphs with Euclidean edges. The geodesic metric, developed in Section~\ref{d_G def}, is easily constructed once a concrete notion of a path is defined. The resistance metric, in contrast, requires decidedly more work and is developed in Section~\ref{d_R def}.


\subsection{Notation and terminology}\label{s:not-term}

Let $\mathcal G=(\mathcal V, \mathcal E, \{\varphi_e\}_{e\in\mathcal E})$ be a  graph with Euclidean edges. To stress the dependence on $\mathcal G$, write $\mathcal V(\mathcal G) \equiv \mathcal V$ and $\mathcal E(\mathcal G)\equiv\mathcal E$. If $u,v \in \mathcal V(\mathcal G)$ are connected by an edge in $\mathcal E(\mathcal G)$, we say they are {\it neighbours} 
and write $u\sim v$.
If $u \in e\in \mathcal E(\mathcal G)$, we let $\underline u,\overline u$ denote the neighbouring vertices which are connected by edge $e$ and ordered so that $\underline u$ corresponds to $\underline{e}$ and $\overline{u}$ corresponds to $\overline{e}$. When $u\in \mathcal V(\mathcal G)$, we define $\overline{u}=\underline{u}=u$. The distinction between $\underline{u}, \overline{u}$ and $\underline e, \overline{e}$ can be seen by noting that  $\underline e,\overline{e}\in \mathbb R$ but $\underline{u},\overline{u}\in \mathcal V(\mathcal G)$.

Let $e\in\mathcal E(\mathcal G)$ and $I\subseteq(\underline e, \overline e)$ be a non-empty interval. Then $\varphi_e^{-1}(I)$ is called a {\it partial edge}, 
its
 two {\it boundaries} 
 correspond to the two-point set
$\varphi_e^{-1}(\overline I \setminus I^o)$, where $\overline I$ is the closure of $I$ and $I^o$ is the open interior of $I$, and its
{\it length} 
is given by the Euclidean length of $I$. Thus the edge $e$ is also a partial edge and its length is denoted $\text{len}(e)$.

Two partial edges are called {\it incident} 
if they share a common boundary in $\mathcal G$.
 A {\it path} 
 connecting two distinct points $u, v\in \mathcal G$ is denoted $p_{uv}$ and given by an alternating sequence $u_1, e_1, u_2, e_2,\ldots, u_n, e_n, u_{n+1}$,
where  $u_1,\ldots,u_{n+1}\in\mathcal G$ are pairwise distinct, $u_1 = u$, $u_{n+1}=v$, and  $e_1, e_2, \ldots, e_n$ are non-overlapping partial edges such that each $e_i$ has boundary $\{u_i, u_{i+1}\}$.
Moreover, the {\it length of $p_{uv}$} 
is denoted  $\text{len}(p_{uv})$ and defined as the sum of the lengths of $e_1, e_2, \ldots, e_n$.


\subsection{Geodesic metric
}\label{d_G def}

For a graph with Euclidean edges $\mathcal G$, the {\it geodesic distance} 
is defined for all $u,v\in\mathcal G$ by
\begin{align}
    \label{geodesic construction}
    d_{G,\mathcal G}(u,v) = \inf\{ \text{len}(p_{uv})
    \}
\end{align}
where the infimum is over all paths connecting $u$ and $v$.
Using the consistency requirement given in Definition \ref{def.gwee}\eqref{def.gwee4}, the following proposition is easily verified.

\begin{prop}\label{thm: basic properties of geodesic metric}
    If $\mathcal G$ is a graph with Euclidean edges, then $d_{G,\mathcal G}(u,v)$
    is a metric over $u,v\in \mathcal G$ satisfying the following.
    \begin{itemize}
        \item\label{geodesic metric definition i} Restricting $d_{G,\mathcal G}$ to $\mathcal V(\mathcal G)$ results in the standard weighted shortest-path graph metric with edge weights given by $\text{len}(e)$.
        \item\label{geodesic metric definition ii} $d_{G,\mathcal G}$ is an extension of the Euclidean metric on each edge $e\in \mathcal E(\mathcal G)$ induced by the bijection $\varphi_e$. That is, $d_{G,\mathcal G}(u,v)= |\varphi_e(u) - \varphi_e(v)|$ whenever  $u,v\in e\in \mathcal E(\mathcal G)$.
    \end{itemize}
\end{prop}


\subsection{Resistance metric
}\label{d_R def}

The resistance metric typically refers to a distance derived from electrical network theory  on the vertices of a finite or countable graph with each edge representing a resistor with a given conductance, see e.g.\ \cite{JorgensenPearse:10} and the references therein. By definition, the resistance between two vertices $u$ and $v$ is the voltage drop when a current of one ampere flows from $u$ to $v$. For a graph with Euclidean edges $\mathcal G$, there are two reasons why it is natural to consider an extension of the resistance metric, defined on just the vertices and edge conductance given by inverse edge length, to the continuum of edge points and vertices of $\mathcal G$. The first reason is purely mathematical: The resulting metric solves the degeneracy problem found in Theorem~\ref{ea:forbidden subparagraph}. Second,
resistance  may be a natural metric for applications associated with flow and travel time across street networks: For example, the total inverse resistance of resistors in parallel is equal to the sum of their individual inverse resistances; correspondingly, multiple pathways engender better flow.

In developing this extension, we take a somewhat non-standard approach and define a metric over $\mathcal G$ with the use of an auxiliary random field $Z_{\mathcal G}$ with index set $\mathcal G$. The resulting metric 
is then \textit{defined} to be the variogram of $Z_{\mathcal G}$:
\begin{equation}\label{e:drGuvdef}
    d_{R,\mathcal G}(u,v) := \text{var}(Z_{\mathcal G}(u) - Z_{\mathcal G}(v)),\qquad u,v\in\mathcal G.
\end{equation}
Propositions~\ref{classical effective resistance} and \ref{splitting and merging edges} below show that $d_{R,\mathcal G}$ does in fact give the natural extension of the electrical network resistance metric: $d_{R,\mathcal G}$ evaluated on any additional edge points will result in the same metric that would be obtained on the resulting discrete electrical network.

Before presenting the formal construction of $Z_{\mathcal G}$ and our results, we give a brief outline.
The form of $Z_{\mathcal G}$ will be defined as a finite sum of independent zero-mean Gaussian random fields:
\begin{equation}\label{ea: def of Z_G}
    Z_{\mathcal G}(u):=Z_\mu(u) +\!\sum_{e\in\mathcal E(\mathcal G)}\!Z_e(u),\qquad u\in\mathcal G.
\end{equation}
The field $Z_\mu$ is characterized by a multivariate Gaussian vector $(Z_{\mu}(v);v\in\mathcal V(\mathcal G))$ whose covariance matrix is related to the so-called graph Laplacian in electrical network theory; this vector is linearly interpolated across the edges so that $Z_\mu(u)$ is defined for all points $u\in\mathcal G$. For each $e\in \mathcal E$, the random field  $Z_e$ is only defined to be non-zero on edge $e$ and  $Z_e(u)=B_e(\varphi_e(u))$ if $u\in e$ or $u$ is a boundary point of $e$, where $B_e$ is an independent Brownian bridge defined over $[\underline e,\overline e]$. Although the construction of $Z_{\mathcal G}$ appears ad-hoc, we will show that the variogram of the resulting random field $Z_{\mathcal G}$ results in the continuum extension of the resistance metric found in electrical network theory.

%

\subsubsection{Construction of  $Z_{\mu}$}

The random field $Z_\mu$ is constructed via analogy to electrical network theory and using the following ingredients. We view each edge in $\mathcal G$ as a resistor with conductance function  $c:\mathcal V(\mathcal G)\times \mathcal V(\mathcal G)\to[0,\infty)$ given by
\begin{align}
    c(u,v) &= \begin{cases}
    1/d_{G,\mathcal G}(u,v) &\text{if $u\sim v$}, \\
    0 &\text{otherwise.} \\
    \end{cases} \label{conductance1}
\end{align}
Let  $\mathbb R^{\mathcal V(\mathcal G)}$ denote the vector space of real functions $h$ defined on $\mathcal V(\mathcal G)$; when convenient
we view $h$ as a vector indexed by $\mathcal V(\mathcal G)$.
Also let $u_o\in \mathcal V(\mathcal G)$ be an arbitrarily chosen vertex called the origin; this is only introduced for technical reasons as explained below.
Define $L:\mathcal V(\mathcal G)\times \mathcal V(\mathcal G)\to\mathbb R$ as the function/matrix with coordinates
\begin{equation}\label{def of L in main txt}
    L(u,v)=\left\{
    \begin{array}{llll}
        1+c(u_o) & \mbox{if }u=v=u_o,\\
        c(u) & \mbox{if }u=v\not=u_o,\\
        -c(u,v) &\mbox{otherwise,}
    \end{array}
    \right.
\end{equation}
where $c(u) :=\sum_{v}c(u,v)=\sum_{v\sim u}c(u,v)$ corresponds to the sum of the conductances associated to the edges incident to vertex $u$.
Obviously, $L$ is symmetric and
a simple calculation shows that for $z,w\in\mathbb R^{\mathcal V(\mathcal G)}$,
\begin{equation}\label{ea:zTLw}
    z^TLw=z(u_o)w(u_o)
    \,+\! \frac{1}{2}\sum_{u\sim v}(z(u)-z(v)) c(u,v) (w(u)-w(v)),
\end{equation}
so $z^T L z = 0$ if and only if $z(u)=0$ for all $u\in \mathcal
V(\mathcal G)$.  Thus $L$ is (strictly) positive definite 
with (strictly) positive definite matrix inverse $L^{-1}$. Notice that the matrix $L$ is similar to what would
be called the \lq\lq Laplacian matrix"  from electrical network theory, see e.g.\ \cite{Kigami:03} and \cite{JorgensenPearse:10}, except that $L$ has the
additional $1$ added at $(u_o, u_o)$. The role of the origin $u_o$ is to make
$L$ (strictly) positive definite, but the resistance metric will be
shown to be invariant to this choice and have the correct form
(see Proposition~\ref{classical effective resistance} below).

Now, the random field $Z_\mu$ is simply defined by linearly
interpolating a collection of Gaussian random variables associated with
the vertices $\mathcal V(\mathcal G)$: Let $v_1, v_2,
\ldots, v_n$ denote the vertices in $\mathcal V(\mathcal G)$
and define $Z_\mu$ at these vertices by
\begin{equation}
    (Z_\mu(v_1), \ldots, Z_\mu(v_n))^T \sim \mathcal N(0,L^{-1}).
\end{equation}
To define the value of $Z_\mu(u)$ at any point $u\in\mathcal G$ we interpolate across
each edge as follows
\begin{equation}
    Z_\mu(u)= (1-d(u))Z_\mu(\underline u) + d(u) Z_\mu(\overline u)
\end{equation}
where  $d(u)$ denotes the distance of $u$ from $\underline u$ as a proportion of the length of the edge containing $u$, formally given by
\begin{equation}\label{ea: definition of d(u)}
    d(u)= \begin{cases}
        d_{G,\mathcal G}(u, \underline u)/d_{G,\mathcal G}(\underline u, \overline u) & \text{if $u\not\in \mathcal V(\mathcal G)$}, \\
        0 & \text{otherwise.}
    \end{cases}
\end{equation}
Notice that the covariance function $R_\mu(u,v):=  \text{cov}(Z_\mu(u), Z_\mu(v))$  can be computed explicitly: For any $u,v\in\mathcal G$,
\begin{align}
    R_\mu(u,v)
    &= d(u)d(v)L^{-1}(\overline u,\overline v) +  [1-d(u)][1-d(v)]L^{-1}(\underline u,\underline v) \nonumber\\
    &\qquad + d(u)[1-d(v)]L^{-1}(\overline u,\underline v) + [1-d(u)]d(v)L^{-1}(\underline u,\overline v). \label{ea:R_mu in main txt}
\end{align}

%

\subsubsection{Construction of  $Z_{e}$} 

The definition of $Z_\mu$ in the previous section used explicitly an analogy to electrical network theory.  So it should come as no surprise that the variogram of $Z_\mu$ gives something related to the resistance metric. However, this will only be true at the vertices. What we want is the electrical network property to hold for all points of $\mathcal G$ without the necessity of re-computing the  matrix $L$ for additional edge points. By simply adding Brownian bridge fluctuations over each edge, this turns out to give the right amount of variability.

To formally define a Brownian bridge process over each edge $e\in\mathcal E(\mathcal G)$, we use the edge bijection $\varphi_e$ which identify points on $e$ with points in the interval $(\underline e,\overline
  e)\subset \mathbb R$.  For all $e\in\mathcal E(\mathcal G)$, let
$B_e$ denote mutually independent Brownian bridges which are independent of $Z_{\mu}$, where $B_e$  is defined on $[\underline
  e,\overline e]$ so that $B_e(\underline e) = B_e(\overline e) = 0$.
  For any $u\in\mathcal G$, we define
\begin{equation}
Z_e(u)= \begin{cases}
    B_e(\varphi_e(u)) & \text{ if $u\in e$}, \\
    0 & \text{ otherwise}.
\end{cases}
\end{equation}
Letting $R_e(u,v)=  \text{cov}(Z_e(u), Z_e(v))$, we have for any $u,v\in \mathcal G$,
\begin{equation}\label{ea:R_e given in main txt}
    R_e(u,v)=\begin{cases}
    \bigl[d(u)\wedge d(v)-d(u)d(v)\bigr]d_{G,\mathcal G}(\underline u,\overline u) & \text{ if $u,v\in e$},\\
    0 & \text{ otherwise.}
    \end{cases}
\end{equation}
Note that the covariance function $R_{\mathcal G}$ for the random field $Z_{\mathcal G}$, defined in \eqref{ea: def of Z_G}, satisfies
\begin{equation}\label{ea:cov fun of Z_G}
    R_{\mathcal G}(u,v)=R_\mu(u,v)+\sum_{e\in\mathcal E(\mathcal G)}R_e(u,v),\qquad u,v\in\mathcal G.
\end{equation}

%

\subsubsection{Properties of $d_{G,\mathcal G}$ and $d_{R,\mathcal G}$}
\label{props of d_G and d_R}

The following Proposition~\ref{classical effective resistance} shows that $d_{R,\mathcal G}$ is indeed the extension of the classical effective resistance on electrical networks and is invariant to the choice of origin $u_o$ (used in the construction of $L$ in \eqref{def of L in main txt}). Further, Proposition~\ref{splitting and merging edges} shows that $d_{R, \mathcal G}$ is invariant to the addition of vertices and removal of vertices with degree two. Finally, Proposition~\ref{relation btwn dG and dr} characterizes $d_{R,\mathcal G}$ via an associated infinite dimensional reproducing kernel Hilbert space. The proofs of the propositions are given in Appendix~\ref{a:2}.

\begin{prop}
\label{classical effective resistance}
For a graph with Euclidean edges $\mathcal G$,
     $d_{R,\mathcal G}$ is a metric, it is invariant to the choice of origin $u_o$, and it simplifies to the classic (effective) resistance metric over the vertices when $\mathcal G$ is considered to be an electrical network with nodes $\mathcal V(\mathcal G)$, resistors given by the edges $e\in \mathcal E(\mathcal G)$, and conductances given by $1/\text{len}(e)$ for $e\in \mathcal E(\mathcal G)$.
\end{prop}

An important property of the geodesic metric on graphs with Euclidean edges is that distances are, in some sense, invariant to the replacement of an edge by two new edges merging at a new degree 2 vertex. This is illustrated in Figure \ref{two cycles isometrically equiv} where it is clear that geodesics are the same for the left-most graph and the middle graph (when the edge lengths are scaled so the circumferences are equal) regardless of the fact that the left-most graph has more vertices and edges.

Perhaps surprisingly, this important property also holds for $d_{R,\mathcal G}$. To state the result, we need to be precise about what it means to add a vertex on an edge and correspondingly remove a degree 2 vertex (merging the corresponding incident edges). The operations will be generically referred to as splitting and merging: For $u\in e\in\mathcal E(\mathcal G)$, define the partial edges $\underline u u=\{\varphi_e^{-1}(t):\underline e<t<\varphi_e(u)\}$ and
$u\overline u =\{\varphi_e^{-1}(t):\varphi_e(u)<t<\overline e\}$, and
partition  $e=\{\underline u u\}\cup \{u\} \cup\{u\overline u \}$. Then
the operation of
{\it splitting an edge $e\in\mathcal E(\mathcal G)$ at $u\in e$} results in a new graph $\mathcal G_{\text{split}}$ with Euclidean edges which is obtained by
adding $u$ to $\mathcal V(\mathcal G)$ and replacing $e\in \mathcal E(\mathcal G)$ with new edges $\underline u u$ and $u\overline u$.
The operation of {\it merging two edges $e_1, e_2\in\mathcal E(\mathcal G)$  which are incident to a degree two vertex $v\in\mathcal V(\mathcal G)$} results in a new graph with Euclidean edges $\mathcal G_{\text{merge}}$ simply obtained by removing $v$ from $\mathcal V(\mathcal G)$ and replacing $e_1, e_2\in \mathcal E(\mathcal G)$ with the single merged edge given by $e_1\cup\{v\}\cup e_2$.

Clearly, $\mathcal G$, $\mathcal G_{\text{merge}}$ and $\mathcal G_{\text{split}}$ are equal as point sets. It is also clear that the geodesic metric is
{\it invariant to splitting edges and merging edges at degree two vertices}
in the sense that
\[
d_{G,\mathcal G}(u,v) = d_{G,\mathcal G^\prime}(u,v)
\]
for all $u,v\in\mathcal G$ whenever $\mathcal G^\prime$ is obtained from $\mathcal G$ by a finite sequence of edge splitting operations and edge merging operations which meet at a degree two vertex. The following theorem shows this property also holds for the resistance metric.

\begin{prop}\label{splitting and merging edges}
For a graph with Euclidean edges $\mathcal G$, the resistance and geodesic metrics $d_{R,\mathcal G}$ and $d_{G,\mathcal G}$ are invariant to splitting edges and merging edges at degree two vertices (so long as the resulting graph satisfies the conditions of Definition \ref{def.gwee}).
\end{prop}

Propositions~\ref{classical effective resistance} and \ref{splitting and merging edges}  show that $d_{R, \mathcal G}$ is the appropriate extension of the classic resistance metric over finite nodes of an electrical network to the continuum of edge points over a graph with Euclidean edges. The next proposition, also analogous to results from electrical network theory, illustrates how multiple pathways between points of $\mathcal G$ leads to a reduction of $d_{R,\mathcal G}$ compared with $d_{\mathcal G,\mathcal G}$.  

\begin{prop}
\label{relation btwn dG and dr}
For any graph with Euclidean edges $\mathcal G$, we have 
\begin{equation}
\label{eq: dr <= dg}
    d_{R,\mathcal G}(u,v) \leq d_{G,\mathcal G}(u,v),\,\qquad\text{$u,v\in\mathcal G$,}
\end{equation}
with equality if and only if $\mathcal G$ is a Euclidean tree. 
If $\mathcal G$ is a Euclidean cycle with circumference $\omega = \sum_{e\in \mathcal E(\mathcal G)}\text{len}(e)$ then
\begin{equation}
\label{eq: form of d_R when G is a cycle}
d_{R,\mathcal G}(u,v) = d_{G,\mathcal G}(u,v) - \frac{{d_{G,\mathcal G}(u,v)^2}}{\omega},\,\qquad\text{$u,v\in\mathcal G$.}
\end{equation}
\end{prop}

The fact that $d_{R,\mathcal G}(u,v) = d_{G,\mathcal G}(u,v)$ if and only if $\mathcal G$ is a Euclidean tree suggests that $d_{R,\mathcal G}$ can be viewed not only as an extension of the vertex (effective) resistance as established in Proposition~\ref{classical effective resistance} but also as an extension of $d_{G,\mathcal G}$ on trees to general graphs. Instead of extending the shortest path property of the geodesic metric, the resistance metric extends the validity of the covariance models given in Theorem~\ref{main thm: iso cov funs}: from $d_{G,\mathcal G}$ on trees to $d_{R,\mathcal G}$ on general graphs (noting that Theorem~\ref{ea:forbidden subparagraph} implies both properties can not be simultaneously extended to general graphs). 

Notice that the quadratic term in \eqref{eq: form of d_R when G is a cycle}, for the Euclidean cycles, explicitly quantifies how multiple paths leads to a reduction in (effective) resistance. Moreover, \eqref{eq: form of d_R when G is a cycle} can be re-arranged, using the fact that $\omega = \text{len}(p_{uv}) + \text{len}(\widetilde p_{uv}) $ where $p_{uv}$ denotes the shortest path from $u$ to $v$ and $\widetilde p_{uv}$ denoting the longer path connecting $u$ to $v$, to obtain
\[
d_{R,\mathcal G}(u,v) = \Big(\,{ \text{len}(p_{uv})^{-1} + \text{len}(\widetilde p_{uv})^{-1}} \Big)^{-1}.
\]
In particular, $d_{R,\mathcal G}(u,v)$ is a function of both path lengths, strictly smaller than each, combined through what is called {\it parallel reduction} in electrical network theory.

We remark that \eqref{eq: form of d_R when G is a cycle} allows us to write any covariance model $C(d_{R, \mathcal G}(u,v))$ on a Euclidean cycle $\mathcal G$ in terms of the geodesic metric $d_{G,\mathcal G}$. Combined with Theorem~\ref{main thm: iso cov funs} we conclude that $C(t - t^2 / \omega)$
is strictly positive definite on the circle of radius $\omega / (2\pi)$ for every non-constant completely monotonic function $C:[0,\infty)\to \mathbb R$. These results can be compared with the literature on isotropic auto-covariance models on the circle.  
For example, in the case $C(t) = \exp(-\beta t)$ the positive definiteness of the auto-covariance $\exp(-\beta (t - t^2 / (2\pi)))$ with respect to the geodesic distance on $\mathbb S^1$ agrees with the conclusions of P\'olya's theorem on the circle (see e.g.\ Theorem~4 in \cite{gneiting1998b}).


Our final result on the resistance metric, although stated last and verified in Appendix~\ref{a:1}, gives the above three propositions as near corollaries and does so by characterizing the reproducing kernel Hilbert space of functions over $\mathcal G$ which is associated to the Gaussian random field $Z_\mathcal G$ (see e.g.\ \cite{wahba1990}). To state the result we need some notation for functions defined over $\mathcal G$. For $f:\mathcal G\to \mathbb R$ and $e\in\mathcal E(\mathcal G)$, we let $f_e:[\underline e, \overline{e}]\to \mathbb R$ denote the restriction of $f$ to $e$, interpreted as a function of the interval $[\underline e, \overline{e}]$. If $f_e$ has a derivative Lebesgue almost everywhere, we denote this by $f_e^{\,\prime}$; recall that the existence of $f_e^{\,\prime}$ is equivalent to absolute continuity of $f_e$.

 \begin{defn}\label{ea:diriclet form def}
     For a graph with Euclidean edges $\mathcal G$ and an arbitrarily chosen origin $u_o\in\mathcal V(\mathcal G)$, let
     $\mathcal F$ be the class of   functions $f:\mathcal G\to \mathbb R$ which are continuous with respect to $d_{G,\mathcal G}$ and for all $e\in \mathcal E(\mathcal G)$, $f_e$ is absolutely continuous and $f_e^\prime \in L^2([\underline e, \overline e])$.
     In addition,
     define the following quadratic form on $\mathcal F$:
     \begin{align}\label{def of inner product}
         \langle f, g\rangle_{\mathcal F} &:= f(u_o)g(u_o) \,+\, \sum_{e\in \mathcal E(\mathcal G)}\, \int_{\underline e}^{\overline e} f_e^{\,\prime} (t) g_e^{\,\prime}(t)\,\mathrm dt.
     \end{align}
 \end{defn}

 \begin{prop}
 \label{the RKHS associated with Z_G}
    Let $\mathcal G$ be a graph with Euclidean edges with origin $u_o\in\mathcal V(\mathcal G)$. Then the space $(\mathcal F, \langle \cdot, \cdot \rangle_{\mathcal F})$ is an infinite dimensional Hilbert space with reproducing kernel $R_{\mathcal G}(u,v)$, given in \eqref{ea:cov fun of Z_G}, and resistance metric $d_{R,\mathcal G}(u,v)$, given in \eqref{e:drGuvdef}, satisfying
     \begin{align}
         d_{R,\mathcal G}(u,v)
          &= \sup_{f\in \mathcal F} \big\{(f(u) - f(v))^2:  \| f \|_{\mathcal F}\leq 1   \big\}, \label{the RKHS inner product associated with Z_G} \\
         R_{\mathcal G}(u,v) &= 1 + \big\{ d_{R,\mathcal G}(u,u_o) + d_{R,\mathcal G}(v,u_o) - d_{R,\mathcal G}(u,v) \big\}/2, \label{dR to R}
     \end{align}
     for all $u,v \in \mathcal G$.
 \end{prop}

Notice that \eqref{dR to R} illustrates how the additional $1$, added to $c(u_o)$ in \eqref{def of L in main txt}, translates to the dependence of $R_{\mathcal G}(u,v)$ on $u_o$. This was done to ensure the invertibility of $L$ but the effect of which is canceled in the variogram of $Z_{\mathcal G}$. Moreover, \eqref{dR to R} also illustrates the connection with classic Brownian motion on $\mathbb R$. For example, if $\mathcal G$ has vertices $0$ and $1$ connected by a single edge $e=(0,1)$, $\varphi_e$ is the identity, and $u_0=0$, then Propositions~\ref{relation btwn dG and dr} and \ref{the RKHS associated with Z_G} shows that $R_{\mathcal G}(u,v)=1+(|u|+|v|-|u-v|)/2$ which, up to a overall constant is precisely the covariance function of Brownian motion.

%
%

\section{Hilbert space embedding of $d_{G,\mathcal G}$ and $d_{R,\mathcal G}$}\label{ea: Hilbert space embedding section}

This section proves the key Hilbert space embedding result given in Theorem~\ref{ea: main embedding for dR}. For this we first need to recall a theorem by \cite{schoenberg1935remarks,schoenberg1938metric} on relating Hilbert spaces and positive definite functions and establish a new theorem on embedding $1$-sums of distance spaces. The exposition of both of these results are kept as general as possible, since they hold for arbitrary distance spaces.

Recall that $(X,d)$ is called a {\it distance space} if $d(x,y)$ for $x,y\in X$ is a
{\it distance}
on $X$, i.e., $d$ satisfies all the requirements of a metric with the possible exception of the triangle inequality.  Let $ \mathrm{Range}(X,d)$ $=\{d(x,y):x,y\in X\}$.

\begin{defn}
    Let $(X, d)$ be a distance space and $g:\mathrm{Range}(X,d)\to[0,\infty)$ a function. Then $(X, d)$ is said to have a
    {\rm $g$-embedding into a Hilbert space} 
    $(H, \| \cdot \|_H )$, denoted
    $(X,d)\overset g\hookrightarrow H$, if 
    there exists a map $\varphi\colon X\to H$ which satisfies \[
    g(d(x,y)) = \|\varphi(x) - \varphi(y)  \|_H
    \]
    for all $x,y\in X$. The special case when $g$ is the identity map is denoted $(X,d)\overset{id}\hookrightarrow H$.
\end{defn}

The following fundamental theorem shows the connection between Hilbert space embeddings and positive semi-definite functions; it follows from \cite{schoenberg1935remarks,schoenberg1938metric}. This turns out to be an extremely useful tool, both for constructing positive semi-definite functions and for proving the existence of Hilbert space embeddings.

\begin{thm}
\label{Schoenberg's thm}
    Let $(X,d)$ be a distance space and $x_0$ an arbitrary member
    of $X$. The following statements are equivalent.
    \begin{enumerate}[{\rm (I)}]
        \item\label{thm itm: embedding} $(X,d)\overset{id}\hookrightarrow H$ for some Hilbert space $H$.
        \item\label{thm itm: characterize p V} $d(x, x_0)^2 + d(y, x_0)^2 - d(x, y)^2$ is positive semi-definite over $x,y \in X$.
        \item\label{thm itm: characterize exp V} For every $\beta >0$,
          the function $\exp(-\beta d(x,y)^2)$ is positive
          semi-definite over $x,y \in X$.
        \item\label{thm itm: negtave type} The inequality
          $\sum_{k,j=1}^n c_k c_j d(x_k,x_j)^2\leq 0$ holds for every
          $x_1,\ldots, x_n\in X$ and $c_1,\ldots, c_n\in\mathbb R$ for
          which $\sum_{k=1}^n c_k = 0$.
    \end{enumerate}
\end{thm}

It is common (in~\cite{wells1975embeddings}
for example) to call any distance space $(X,d)$ which satisfies
condition (\ref{thm itm: negtave type}) a
{\it distance of negative type}. 
In the geostatistical literature, however, if $d$ satisfies
(\ref{thm itm: negtave type}), then $d^2$ is said to be a
{\it generalized covariance function of order $0$}. 
In
particular, for any random field $Z$, condition (\ref{thm itm: negtave
  type}) is a necessary property of the variogram ${d(u,v)}^2=
\text{var}(Z(u) - Z(v))$.

The last concept needed to show Theorem \ref{ea: main embedding for dR} deals with the notion of the $1$-sum of two distance spaces \citep{deza1996geometry}. This operation  allows us to construct new distance spaces (which are root-embeddable) by stitching  multiple root-embeddable distance spaces together.

\begin{defn}\label{1sum def}
    Suppose $(X_1, d_1)$ and $(X_2, d_2)$ are two distance spaces such that $X_1\cap X_2 = \{ x_0 \}$. Then the
    {\rm $1$-sum} 
    of  $(X_1, d_1)$ and $(X_2, d_2)$ is the distance space $(X_1\cup X_2, d)$ defined by
    \begin{equation}\label{1sum def eqn for d}
    d(x,y)=
    \begin{cases}
        d_1(x,y) & \text{if $x,y \in X_1$}, \\
        d_2(x,y) & \text{if $x,y \in X_2$}, \\
        d_1(x,x_0) + d_2(x_0,y) & \text{if $x \in X_1$ and $y \in X_2$}.
    \end{cases}
    \end{equation}
\end{defn}


The key fact about $1$-sums, for our use, is summarized in the following theorem. We omit the proof and simply note that it can be found in~\cite{deza1996geometry} (Proposition 7.6.1) with slightly different nomenclature. 


\begin{thm}\label{thm sum closure}
    Suppose $(X_1, d_1)$ and $(X_2, d_2)$ are two distance spaces such that $X_1\cap X_2 = \{ x_0 \}$. If $(X_1, d_1)\overset{\text{\tiny$\sqrt{}$}}\hookrightarrow H_1$ and
    $(X_2, d_2)\overset{\text{\tiny$\sqrt{}$}}\hookrightarrow H_2$
    for two Hilbert spaces $H_1$ and $H_2$,  then there exists a Hilbert space $H$ such that $(X_1\cup X_2, d)\overset{\text{\tiny$\sqrt{}$}}\hookrightarrow H$ where $(X_1\cup X_2,d)$ is the
    {\rm $1$-sum} 
    of $(X_1, d_1)$ and  $(X_2, d_2)$.
\end{thm}

We are now ready to prove Theorem \ref{ea: main embedding for dR}, from Section \ref{ea: summary of main results}, on the Hilbert space embedding of $d_{R,\mathcal G}$ and $d_{G,\mathcal G}$.

\begin{proof}[\small{PROOF OF THEOREM \ref{ea: main embedding for dR}}]  

Suppose $\mathcal G$ is a graph with Euclidean edges.
By 
\eqref{e:drGuvdef}, we trivially have
\[
    d(u,v)^2  = \text{var}(Z_{\mathcal G}(u) - Z_{\mathcal G}(v))
\]
where $d(u,v):=\sqrt{d_{R,\mathcal G}(u,v)}$. The fact that $d(u,v)^2$ is a variogram implies that condition \eqref{thm itm: negtave type} holds (from Schoenberg's result stated in Theorem~\ref{Schoenberg's thm}). Since \eqref{thm itm: embedding} $\Longleftrightarrow$ \eqref{thm itm: negtave type} we have that $(\mathcal G, d)\overset{id}\hookrightarrow H$ for some Hilbert space $H$, and hence
$(\mathcal G, d_{R,\mathcal G})\overset{\text{\tiny $\sqrt{}$}}\hookrightarrow H$ as was to be shown.

For the geodesic metric, first assume $\mathcal G$ forms a tree graph.
In this case, Proposition~\ref{relation btwn dG and dr} implies $d_{G,\mathcal G} = d_{R,\mathcal G}$ and therefore $(\mathcal G,  d_{G,\mathcal G})\overset{\text{\tiny $\sqrt{}$}}\hookrightarrow H$ by the corresponding result for $d_{R,\mathcal G}$.
Second, assume $\mathcal G$ forms a cycle graph (such as the left two graphs of Figure~\ref{two cycles isometrically equiv}). For some constant $\lambda>0$, there clearly exists a metric isometry between $(\mathcal G,  \lambda d_{G,\mathcal G})$ and the unit circle $\mathbb S^1$ equipped with the great circle metric $d_{\mathbb S^1}$.
Since $\exp(-\beta d_{\mathbb S^1}(x,y))$ is positive semi-definite over ${\mathbb
  S^1}\times {\mathbb S^1}$ for all $\beta>0$ (see \cite{gneiting2013strictly}, for example), the function $\exp(-\beta d_{G,\mathcal G}(u,v))$ is positive semi-definite over $\mathcal G\times\mathcal G$ for all $\beta>0$. Now, setting $d:=\sqrt{d_{G,\mathcal G}}$, the equivalence \eqref{thm itm: embedding} $\Longleftrightarrow$ \eqref{thm itm: characterize exp V} in Theorem \ref{Schoenberg's thm} implies  $(\mathcal G, d)\overset{id}\hookrightarrow H$, hence $(\mathcal G, d_{G,\mathcal G})\overset{\text{\tiny $\sqrt{}$}}\hookrightarrow H$ for some Hilbert space $H$. Finally, for the general result where $\mathcal G$ is a $1$-sum of cycles and trees, we simply use Theorem \ref{thm sum closure} to conclude that $(\mathcal G, d_{G,\mathcal G})\overset{\text{\tiny $\sqrt{}$}}\hookrightarrow H$ for some Hilbert space $H$.
\end{proof}

%
%

\section{Isotropic covariance functions with respect to $d_{G,\mathcal G}$ and $d_{R,\mathcal G}$}\label{forbidden subgraph for d_G}

In this section we prove all the results stated in Section~\ref{ea: summary of main results}, with the exception of Theorem \ref{ea: main embedding for dR} proved in the previous section. In some sense, many of the proofs of these results follow easily from Theorem \ref{ea: main embedding for dR} and the seminal work of Schoenberg and von Neumann~\citep{schoenberg1938metric, schoenberg1938monotone, von1941fourier} connecting metric embeddings, Hilbert spaces and completely monotonic functions (see also \cite{gneiting2013strictly}) but we review the necessary results for completeness. The following result characterizes completely monotonic functions on $[0,\infty)$ as positive mixtures of scaled exponentials (see Theorems $2$, $3$, and $3^\prime$ in \cite{schoenberg1938monotone}).

\begin{thm}\label{thm:charcomplmonofuns2}
    The completely monotonic functions on $[0,\infty)$ are precisely those which admit a representation $f(t) = \int_{0}^\infty e^{-t\sigma}\,\mathrm d\mu(\sigma)$, where $\mu$ is a finite positive measure on $[0,\infty)$. Moreover, if $H$ is a Hilbert space and $f$ is a non-constant completely monotonic function on $[0,\infty)$ then $f(\|x-y\|^2_H)$  is (strictly) positive definite over $(x,y)\in H\times H$. 
\end{thm}

\begin{corollary}\label{functional form of the auto-covariance functions we construct in this paper}
If $C:[0,\infty)\to\mathbb R$ is a non-constant and completely monotonic function and $(X,d)$ is a distance space which satisfies
    \[
    (X,d)\overset{\text{\tiny$\sqrt{}$}}\hookrightarrow H,
    \]
    where  $H$ is a Hilbert space, then $C(d(x,y))$ is positive semi-definite over $(x,y)\in X\times X$. If, in addition, $d(x,y) = 0 \Leftrightarrow x=y$ for all $x,y\in X$, then $C(d(x,y))$ is (strictly) positive definite over $(x,y)\in X\times X$.
\end{corollary}

Corollary \ref{functional form of the auto-covariance functions we construct in this paper} follows easily from Theorem \ref{thm:charcomplmonofuns2}, since if $(X,d)\overset{\text{\tiny$\sqrt{}$}}\hookrightarrow H$ for some Hilbert space $H$, then there exists a map $\varphi\colon X \to H$ for which $d(x,y) = \|\varphi(x) - \varphi(y)  \|_H^2$ for all $x,y\in X$.  Then for a non-constant and completely monotonic function $C$ we have $C(d(x,y)) =  C(\|\varphi(x) - \varphi(y)  \|_H^2)$ which is positive semi-definite, via Theorem \ref{thm:charcomplmonofuns2}. If, in addition, $d(x,y) = 0 \Leftrightarrow x=y$ for all $x,y\in X$, then  $\varphi$ maps one-to-one onto its range which implies that $C(d(x,y)) =  C(\|\varphi(x) - \varphi(y)  \|_H^2)$ is strictly positive definite.  This establishes Corollary \ref{functional form of the auto-covariance functions we construct in this paper}. 

Now, we turn to the proofs of the remaining four theorems stated in Section~\ref{ea: summary of main results}: Theorems \ref{main thm: iso cov funs},  \ref{ea:forbidden subparagraph}, \ref{thm: extension for tree graphs}, and \ref{thm: simplified extension for tree graphs}.

\begin{proof}[\small{PROOF OF THEOREM \ref{main thm: iso cov funs}}]
Suppose  $\mathcal G$ is a graph with Euclidean edges and let $C:[0,\infty)\to\mathbb R$ be a non-constant and completely monotonic. The metric properties of both $d_{G,\mathcal G}$ and $d_{R,\mathcal G}$ imply that $d_{G,\mathcal G}(u,v) = 0 \Leftrightarrow u = v$ and $d_{R,\mathcal G}(u,v) = 0 \Leftrightarrow u = v$ for all $u,v\in\mathcal G$.  Theorem~\ref{main thm: iso cov funs} now follows immediately from Theorem~\ref{ea: main embedding for dR} and Corollary~\ref{functional form of the auto-covariance functions we construct in this paper}. 

\end{proof}

\begin{proof}[\small{PROOF OF THEOREM \ref{ea:forbidden subparagraph}}]
By uniformly scaling $d_{G,\mathcal G}$ and possibly selecting new vertices on $\mathcal G$ by edge splitting operations (Proposition~\ref{splitting and merging edges}), 
one can obtain $6$ vertices $u_1,\ldots, u_6$ on $\mathcal G$ which have the following geodesic pairwise distance matrix where $0<t\le r \le 1$:
    \begin{align*}
        \left\{d_{G,\mathcal G}(u_i,u_j)\right\}_{i,j = 1}^6=
        \left(
        \begin{array}{cccccc}
            0 & t & 1 & r+1 & 1 & t+1 \\
            t & 0 & 1-t & r-t+1 & t+1 & 1 \\
            1 & 1-t & 0 & r & 2 t & t \\
            r+1 & r-t+1 & r & 0 & r & r+t \\
            1 & t+1 & 2 t & r & 0 & t \\
            t+1 & 1 & t & r+t & t & 0
        \end{array}
        \right)
    \end{align*}
    corresponding to the following subgraph:
        \begin{figure}[H]
        \centering
        \includegraphics[height=1.65in]{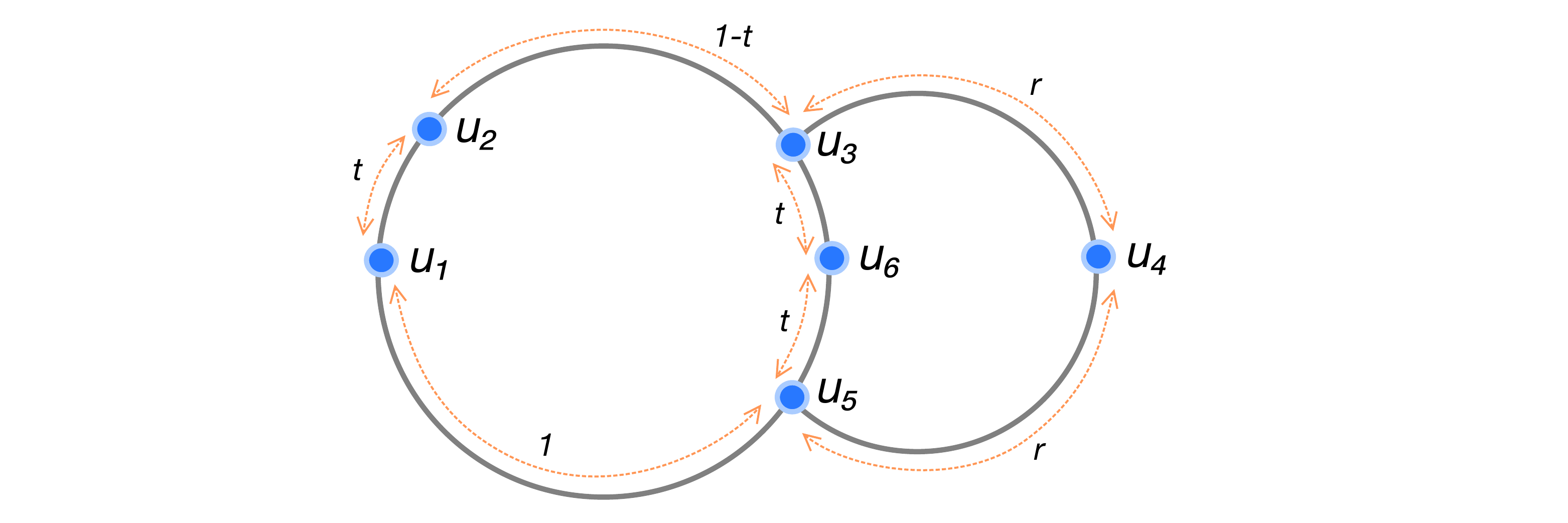}%
    \end{figure}
  \noindent  The values $2t$, $2r$ and $2$  represent the lengths of the three paths connecting vertices $u_3$ and $u_5$, ordered smallest to largest. Notice that in the case $t = r = 1$ one has $u_3 = u_5$ which implies the graph shown above will only have $5$ distinct vertices. 
  Forming the matrix $\Sigma = \frac{1}{2} \{d_{G,\mathcal G}(u_i,u_1)+d_{G,\mathcal G}(u_1,u_j) - d_{G,\mathcal G}(u_i,u_j)\}_{i,j = 2}^6$ gives
    \begin{align*}
        \Sigma
        & =
        \left(
        \begin{array}{ccccc}
    t & t & t & 0 & t \\
    t & 1 & 1 & 1-t & 1 \\
    t & 1 & r+1 & 1 & 1 \\
    0 & 1-t & 1 & 1 & 1 \\
    t & 1 & 1 & 1 & t+1 \\
\end{array}
        \right).
    \end{align*}
    Setting $\boldsymbol\xi =  (-1, -\xi, \xi, -1, 1)^T$, we have
    $\boldsymbol\xi^T\Sigma \boldsymbol\xi = \xi (r \xi-2 t)$, so  $\boldsymbol\xi^T\Sigma \boldsymbol\xi< 0$ when $0<\xi< 2t/r$, implying
     that  $c(u_i,u_j)= \frac{1}{2}(d_{G,\mathcal G}(u_i,u_1) + d_{G,\mathcal G}(u_j, u_1) - d_{G,\mathcal G}(u_i, u_j))$ is not positive semi-definite over $\{u_1,\ldots, u_6 \}$. Then Theorem \ref{Schoenberg's thm} gives the existence of a $\beta>0$ such that $\exp(-\beta d_{G,\mathcal G}(u,v))$ is \textit{not} positive semi-definite over $\{u_1,\ldots, u_6 \}$.
\end{proof}

\begin{proof}[\small{PROOF OF THEOREMS \ref{thm: extension for tree graphs} and \ref{thm: simplified extension for tree graphs}}]

Let $\mathcal G$ be a Euclidean tree with $m$ leaves, where $m\geq 3$. Set $n = \lceil \frac{m}{2} \rceil$ and let $(\mathbb R^n, d_1)$ denote the usual $\ell_1$ metric space so that $d_1(x,y)=\sum_{i=1}^n |x_i - y_i|$ 
for $x,y\in\mathbb R^n$. 

First, we show that $(\mathcal G, d_{G,\mathcal G})\overset{id}\hookrightarrow (\mathbb R^n, d_1)$ and $(\mathcal G, d_{R,\mathcal G})\overset{id}\hookrightarrow (\mathbb R^n, d_1)$.   By well known properties of tree graphs, $(\mathcal V(\mathcal G), d_{G,\mathcal G})\overset{id}\hookrightarrow (\mathbb R^n, d_1)$, see e.g.\ Proposition 11.1.4 in \cite{deza1996geometry}.
To extend this embedding from 
$\mathcal V(\mathcal G)$ to all points in $\mathcal G$ it will be sufficient, by Theorem 3.2.2 in \cite{deza1996geometry}, to show $(\mathcal U, d_{G,\mathcal G})\overset{id}\hookrightarrow (\mathbb R^n, d_1)$ for any finite subset $\mathcal U\subset \mathcal G$. Since $(\mathcal U\cup\mathcal V(\mathcal G) , d_{G,\mathcal G})$ is also isometric to a tree graph with $m$ leaves, we have that $(\mathcal U\cup\mathcal V(\mathcal G), d_{G,\mathcal G})\overset{id}\hookrightarrow (\mathbb R^n, d_1)$. This implies that $(\mathcal U,d_{G,\mathcal G})$ embeds into  $(\mathbb R^n, d_1)$, via restriction, as was to be shown. Now, since $\mathcal G$ is a Euclidean tree, Proposition~\ref{relation btwn dG and dr} implies $d_{G,\mathcal G}=d_{R,\mathcal G}$. Therefore, we also have $(\mathcal G, d_{R,\mathcal G})\overset{id}\hookrightarrow (\mathbb R^n, d_1)$.

Second,  Theorem 3.1 of \cite{cambanis1983charfuns} implies $C(d_1(x,y))$ is positive semi-definite over $x,y\in \mathbb R^n$, thus proving Theorem~\ref{thm: extension for tree graphs} by $(\mathcal G, d_{G,\mathcal G})\overset{id}\hookrightarrow (\mathbb R^n, d_1)$ and $(\mathcal G, d_{R,\mathcal G})\overset{id}\hookrightarrow (\mathbb R^n, d_1)$. 
Finally, Theorem 3.2 in \cite{gneiting1998charfuns} establishes Theorem~\ref{thm: simplified extension for tree graphs}.
\end{proof}

%
%

\section{Restricted covariance function properties}\label{sec.forbidden}

The restriction on the parameter $\alpha$ in 
Table \ref{ea:table of acf} agrees with results for similar families
of covariance functions for isotropic random fields on the
$d$-dimensional sphere $\mathbb S^d$
\citep{gneiting2013strictly}. This may be no surprise, since a
Euclidean cycle is similar to the circle $\mathbb S^1$; in fact, all the covariance functions in Table~1 of \cite{gneiting2013strictly} for the circle can be adapted when $\mathcal G$ is a Euclidean cycle.
Below Corollary~\ref{cor-to-thm: restrictions on the variogram} shows that the restriction is in general also needed
when considering a Euclidean tree $\mathcal G$,
noting that if $\mathcal G$ has maximum degree
$n<\infty$, then $\mathcal G$ has a star-shaped subgraph with $n+1$
vertices and $n$ edges. Moreover, Corollary~\ref{cor.nonnegunbsup} shows that there are some quite
severe limitations on the kind of covariance function that are valid
for arbitrary Euclidean trees (and thus also arbitrary graphs with
Euclidean edges).

%

In the following we only consider Euclidean
trees. Then, by Theorem~\ref{relation btwn dG and
  dr}, $d_{G,\mathcal{G}}=d_{R,\mathcal{G}}$ and we use $d_{\cdot,\mathcal{G}}$ as a common
notation for the two metrics. 

\begin{prop}\label{thm: restrictions on the variogram}
If $Z$ is a random field on a Euclidean tree $\mathcal G$ which
contains a star-shaped tree subgraph $\mathcal S_n$ with $n+1$
vertices and $n$ edges, and $\tilde\alpha,\tilde\beta>0$ are numbers
so that $\text{var}(Z_n(u)-Z_n(v))=\tilde\beta
d_{\cdot,\mathcal{G}}(u,v)^{\tilde\alpha}+o(d_{\cdot,\mathcal{G}}(u,v))$
when $d_{\cdot,\mathcal{G}}(u,v)\rightarrow 0$, then $\tilde\alpha\leq
\log(2n/(n-1))/\log(2)$.
\end{prop}
\begin{proof}
Let $u_0$ be the vertex in $\mathcal S_n$ with degree $n$ and consider the variogram $d(u,v)^2 = \text{var}(Z(u) - Z(v))$. By Theorem~\ref{Schoenberg's thm},
$C(u,v)=d(u, u_0)^2 + d(v, u_0)^2 - d(u, v)^2$ is positive semi-definite over $\mathcal S_n\times \mathcal S_n$. For $i=1,\ldots, n$ and $\epsilon >0$ sufficiently small, let $u_{i,\epsilon}$ be the point on the $i^\text{th}$ edge which has $d_{\cdot,\mathcal{G}}$ distance $\epsilon$ from $u_0$. The assumption on $\text{var}(Z(u)-Z(v))$ implies
\begin{align*}
d(u_{i,\epsilon},u_{0})^2 &=\tilde\beta  d_{\cdot,\mathcal{G}}(u_{i,\epsilon},u_0)^{\tilde\alpha} + o\left(d_{\cdot,\mathcal{G}}(u_{i,\epsilon},u_0)\right)= \tilde\beta \epsilon^{\tilde\alpha}+ o\left(\epsilon^{\tilde\alpha}\right) \\
d(u_{i,\epsilon},u_{j,\epsilon})^2 &= \tilde\beta d_{\cdot,\mathcal{G}}(u_{i,\epsilon},u_{j,\epsilon})^{\tilde\alpha}+o\left(d_{\cdot,\mathcal{G}}(u_{i,\epsilon},u_{j,\epsilon})\right) = \tilde\beta 2^{\tilde\alpha}\epsilon^{\tilde\alpha}+ o\left(\epsilon^{\tilde\alpha}\right)
\end{align*}
 when $i\neq j$. Let $\Sigma_\epsilon$ be the $n\times n$ covariance matrix with $(i,j)^\text{th}$ entry
\[(\Sigma_\epsilon)_{i,j}=C(u_{i,\epsilon},u_{j,\epsilon})
= \tilde\beta (2   - 2^{\tilde\alpha})\epsilon^{\tilde\alpha} + \tilde\beta 2^{\tilde\alpha}\epsilon^{\tilde\alpha}\delta_{ij} + o(\epsilon^{\tilde\alpha}),\]
then
\[\Sigma_\epsilon = \tilde\beta (2   - 2^{\tilde\alpha})\epsilon^{\tilde\alpha} \bs 1_n \bs 1_n^T  + \tilde\beta 2^{\tilde\alpha}\epsilon^{\tilde\alpha} I_n  + o(\epsilon^{\tilde\alpha})A\]
 where $I_n$ is the $n\times n$ identity matrix, $\bs 1_n$ is the vector of length $n$ with each coordinate equal to $1$, and $A$ is some $n\times n$ matrix not depending on $\epsilon$. Now,
\begin{align*}
0\le \det(\Sigma_\epsilon) &=\tilde\beta^n 2^{n\tilde\alpha}\epsilon^{n\tilde\alpha}\det\Big( (2^{1-\tilde\alpha}   - 1) \bs 1_n \bs 1_n^T + I_n + o(1)A\Big) \\
&= \tilde\beta^n 2^{n\tilde\alpha}\epsilon^{n\tilde\alpha}\big( (2^{1-\tilde\alpha}   - 1) n  + 1\big)+ o(\epsilon^{n\tilde\alpha}).
\end{align*}
Consequently $(2^{1-\tilde\alpha}   - 1) n  + 1 \geq 0$ as was to be shown.
\end{proof}

\begin{corollary}\label{cor-to-thm: restrictions on the variogram}
  Let $C$ be one of the functions given in Table \ref{ea:table of acf} but with $\alpha$ outside the parameter range, i.e., $\alpha>\frac{1}{2}$ in case of the Mat{\'e}rn class and $\alpha>1$ in case of the other three classes. Then there exists a
  Euclidean tree $\mathcal G$ so that
  $C(d_{\cdot,\mathcal{G}}(u,v))$ is not a
  covariance function.
\end{corollary}

\begin{proof} Suppose $Z_n$ is a random field on a
Euclidean tree $\mathcal G$ which contains a star-shaped graph
$\mathcal S_n$ with $n+1$ vertices and $n$ edges, with 
an isotropic covariance function with radial profile $C$ and $\alpha>0$.

If $C$ is in the Mat{\'e}rn class, let
  \[\tilde\alpha =
  \alpha+1-|\alpha-1|,\qquad \tilde\beta=\frac{\beta^{\alpha+1-|\alpha-1|}\Gamma(|\alpha-1|)2^{|\alpha-1|-\alpha}}
    {\tilde\alpha\Gamma(\alpha)}.\]
  By L'Hospital's Rule, equation 24.56 in
  \cite{schaum-68} and equation 9.6.9 in \cite{abramowitz:stegun:64},
  \begin{align*}
   & \lim_{d_{\cdot,\mathcal{G}}(u,v)\rightarrow0}
    \frac{\text{var}(Z(u)-Z(v))}{d_{\cdot,\mathcal{G}}(u,v)^{\tilde\alpha}}\\
    =&\, \lim_{d_{\cdot,\mathcal{G}}(u,v)\rightarrow0}
    \frac{2(1-\frac{1}{\Gamma(\alpha)2^{\alpha-1}} (\beta d_{\cdot,\mathcal
        G}(u,v))^{\alpha}K_\alpha(\beta d_{\cdot,\mathcal{G}}(u,v))}
    {d_{\cdot,\mathcal{G}}(u,v)^{\tilde\alpha}}\\
    =&\, \lim_{d_{\cdot,\mathcal{G}}(u,v)\rightarrow0}
    \frac{\beta^{\alpha+1}}{\Gamma(\alpha)2^{\alpha-1}\tilde\alpha}
    d_{\cdot,\mathcal{G}}(u,v)^{\alpha-\tilde\alpha+1}
    K_{\alpha-1}(\beta d_{\cdot,\mathcal{G}}(u,v))\\
    =&\, \lim_{d_{\cdot,\mathcal{G}}(u,v)\rightarrow0}\frac{\beta^{\alpha+1-|\alpha-1|}\Gamma(|\alpha-1|)2^{|\alpha-1|-\alpha}}
    {\tilde\alpha\Gamma(\alpha)}
    d_{\cdot,\mathcal{G}}(u,v)^{\alpha-\tilde\alpha+1-|\alpha-1|}=\tilde\beta.
  \end{align*}
  Hence Proposition~\ref{thm: restrictions on the variogram} applies, and
  letting
  $n\rightarrow\infty$ we obtain $\tilde\alpha\le 1$ or equivalently $\alpha\leq\frac{1}{2}$, thus proving the assertion.

If $C$ is in one of the other three classes, it
follows directly from L'Hospital's Rule that the requirement of the
variogram $\text{var}(Z_n(u)-Z_n(v))$ of Theorem~\ref{thm:
  restrictions on the variogram} is satisfied for $\tilde\alpha =
\alpha$ when $\tilde\beta=2\beta$ in case of the power exponential class and
$\tilde\beta=2\beta\xi/\alpha$ in case of the generalized Cauchy or
the Dagum class.  Letting $n\rightarrow\infty$ in
Proposition~\ref{thm: restrictions on the variogram}, we get that
$\alpha\leq1$, thus completing the proof.
\end{proof}

\begin{prop}
  Suppose $(u,v)\to C(d_{\cdot,\mathcal{G}}(u,v))$ is a covariance function
  on a Euclidean tree $\mathcal{G}$ containing a star-shaped tree subgraph
  ${\mathcal S}_n$ with $n\ge2$ edges 
  of length
  larger than or equal to $t_0>0$. For all $t\in(0,t_0]$, we have
  \begin{equation}\label{eq.covineq1}
 -\frac{C(0)}{n-1}\leq C(2t)\leq C(0),\qquad  \frac{n C(t)^2-C(0)^2}{n-1} \le C(0)C(2t).
  \end{equation}
\end{prop}

\begin{proof} 
Denote $e_1,\ldots,e_n$ the edges of ${\mathcal S}_n$, and $u_{n+1}$ their common vertex.
  Let
  $t\in(0,t_0)$ and $u_i\in e_i$ such that
  $d_{\mathcal{\cdot,\mathcal{G}}}(u_{n+1},u_i)=t$ for $i=1,\ldots,n$. Note that
  $d_{\mathcal{\cdot,\mathcal{G}}}(u_i,u_j)=2t$ for $i,j=1,\ldots,n$ and
  $i\not=j$. Let $\Sigma$ denote the $(n+1)\times(n+1)$ matrix
  with the $(i,j)^{\text{th}}$ entry equal to $C(d_{\mathcal{\cdot,\mathcal{G}}}(u_i,u_j))$,
  i.e.,
  \[
  \Sigma_{i,j} =
  \begin{cases}
    C(0) & \text{if } i=j,\\
    C(2t) & \text{if } i\not= j \text{ and } i,j<n+1,\\
    C(t) & \text{otherwise}.
  \end{cases}
  \]
  As $\Sigma$ is a covariance matrix,
   its principal minors are non-negative
 determinants; these are of the form $\det(\Sigma_k)$ with $k\in\{1,\ldots,n\}$ or $\det(\Sigma_k')$ with $k\in\{2,\ldots,n+1\}$; here,
   $\Sigma_k$ denotes a $k\times
  k$ submatrix of $\Sigma$ with the same rows and columns removed and where
  the $(n+1)^{\text{th}}$ row and column have been removed; and
  $\Sigma_k'$ is defined in a similar way
   but where the $(n+1)^{\text{th}}$ row
  and column have not been removed.
  It is easily verified that
  \begin{equation*}\label{eq.covineq1a}
  \det(\Sigma_k) = \left(C(0)-C(2t)\right)^{k-1}\left\{(k-1)C(2t)+C(0)\right\}
  \end{equation*}
   for $k=1,\ldots,n$, and hence either $0\le C(0)=C(2t)$ or both $C(0)>C(2t)$ and
  $(k-1)C(2t)+C(0)\ge0$, implying
   the first inequality in \eqref{eq.covineq1}, where
  we have let $k=n$ to obtain the highest lower bound. Moreover,
  \begin{equation*}\label{eq.covineq2a}
  \det(\Sigma_k') = \left\{C(0)-C(2t)\right\}^{k-2}\left\{C(0)^2 + (k-2)C(2t)C(0) -
    (k-1)C(t)^2\right\}
  \end{equation*}
  for $k=2,\ldots,n+1$, and so either $|C(t)|\le C(0)=C(2t)$ or both $C(0)>C(2t)$ and
    $C(0)^2 + (k-2)C(2t)C(0) - (k-1)C(t)^2\ge0$, implying the second inequality in \eqref{eq.covineq1},
  where we have used
  $k=n+1$ to get the highest lower bound.
\end{proof}

\begin{corollary}\label{cor.nonnegunbsup}
  A function $(u,v)\to C(d_{\mathcal{\cdot,\mathcal{G}}}(u,v))$ which is a covariance
  function on all Euclidean trees has to be non-negative, and
  furthermore either have unbounded support or fulfill $C(t)=0$ for
  all $t>0$.
\end{corollary}

\begin{proof}
  Letting $n\rightarrow\infty$ and $t_0\rightarrow\infty$, the first inequality in \eqref{eq.covineq1} implies
   non-negativity of $C$, and the second inequality in \eqref{eq.covineq1} implies $C(0)C(2t)\geq C(t)^2$ for all
  $t>0$. Thus, if for some
  $t_1>0$ and all $t>t_1$ we have $C(2t)=0$, then $C(t)=0$ for all $t>t_1$, from which it follows by
  induction that $C(t)=0$ for all $t>0$.
\end{proof}



To illustrate the scope of the covariance function restrictions given above it is instructive to consider the variogram suggested on page 205 in \cite{okabe:sugihara:12} in connection to linear networks. If there is a corresponding isotropic covariance function, its radial profile is given by
\[
C(t) = \begin{cases} \beta_0 + \beta_1\beta_2 & \mbox{if }t=0, \\
  \beta_1(\beta_2-t) & \mbox{if }0<t\leq \beta_2\\ 0 & \mbox{if } t>\beta_2,\end{cases}
\]
where $\beta_0, \beta_1, \beta_2>0$ are parameters. As this function has bounded support, by
Corollary~\ref{cor.nonnegunbsup} this cannot be a valid covariance function
on an arbitrary graph with Euclidean edges (or an arbitrary linear
network). However, as remarked in the paragraph proceeding Theorem~\ref{thm: extension for tree graphs}, this does not preclude the positive definiteness of $C(t)$ on a particular {\sl fixed} tree graph. Indeed, Theorem~\ref{thm: simplified extension for tree graphs} can be invoked to imply that when $\mathcal G$ is a Euclidean tree with $m\geq 3$ leaves, then $C(t)^\alpha$ is positive semi-definite (with respect to $d_{R,\mathcal G}=d_{G,\mathcal G}$) for any $\alpha\geq 2 \lceil m/2\rceil - 1$.

\appendix

\section{Proofs}


 \subsection{Proof of Proposition~\ref{the RKHS associated with Z_G}}\label{a:1}

To verify Proposition~\ref{the RKHS associated with Z_G}, recall Definition~\ref{ea:diriclet form def}. We use the notation $I_A$ for an indicator function which is 1 on a set $A$ and 0 otherwise. We need the following lemmas.

  \begin{lemma} 
      $(\mathcal F,\langle \cdot, \cdot\rangle_{\mathcal F})$ is an inner product vector space,  with metric $\|f\|_{\mathcal F}:=\sqrt{\langle f, f\rangle_{\mathcal F}}$ given by
      \begin{equation}\label{ea:metricF}
          \|f\|_{\mathcal F}^2=f(u_o)^2 \,+\, \sum_{e\in \mathcal E(\mathcal G)}\, \int_{\underline e}^{\overline e} f_e^{\,\prime} (t)^2 \,\mathrm dt.
      \end{equation}
  \end{lemma}

  \begin{proof}
      From \eqref{def of inner product} we obtain \eqref{ea:metricF}. Note that $\langle f, f\rangle_{\mathcal F}=0$ implies both $f(u_o)=0$ and for any $e\in\mathcal E(\mathcal G)$, $f_e$ is almost everywhere constant on $e$. The continuity requirement of $f\in \mathcal F$ then implies  $\langle f, f\rangle_{\mathcal F}=0 \Leftrightarrow f=0$.
      Finally, $\langle \cdot, \cdot\rangle_{\mathcal F}$ is clearly symmetric, bilinear and positive semi-definite over $f\in\mathcal F$.
  \end{proof}

 For $f\in\mathcal F$, $u\in\mathcal G$ and $e\in\mathcal E(\mathcal G)$, define
  \begin{equation*}\label{e:bad notation}
      f_\mu(u)=(1-d(u))f(\underline u) + d(u)f(\overline u),
      \qquad
      f_{e,r}(u)=\begin{cases} f(u)-f_\mu(u) & \text{ if $u\in e$},\\ 0 & \text{ otherwise,}\end{cases}
  \end{equation*}
where $d(u)$ is defined in (\ref{ea: definition of d(u)}).
It will be convenient to denote the operations $f\to f_\mu$ and $f \to f_{e,r}$ with operator notation $\mathscr P_\mu:\mathcal F\to \mathcal F$ and $\mathscr P_e:\mathcal F\to \mathcal F$ given by
$\mathscr P_\mu f=f_\mu$ and $\mathscr P_e f=f_{e,r}$.
In addition, the inner product $\langle \cdot, \cdot \rangle_{\mathcal F}$ restricted to the function spaces  $\mathscr P_\mu\mathcal F$ and $\mathscr P_e\mathcal F$ will be denoted $\langle \cdot, \cdot\rangle_\mu = \langle \cdot, \cdot\rangle_\mathcal F\Big|_{\mathscr P_\mu\mathcal F\times \mathscr P_\mu\mathcal F}$ and
$\langle \cdot, \cdot\rangle_{e,r} =
\langle \cdot, \cdot\rangle_\mathcal F\Big|_{\mathscr P_e\mathcal F\times \mathscr P_e\mathcal F}$.

 \begin{lemma}\label{lemma on direct sum} If $\mathcal G$ is a graph with Euclidean edges, then for all $e\in\mathcal E(\mathcal G)$, $\mathscr P_\mu$ and $\mathscr P_e$
     are mutually orthogonal projections and $\mathcal F$ is a direct sum:
     \begin{equation}\label{e:structure}
         \mathcal F=\mathscr P_\mu\mathcal F\oplus\bigoplus_{e\in\mathcal E(\mathcal G)}\mathscr P_e\mathcal F.
     \end{equation}
 \end{lemma}

 \begin{proof} This is straightforwardly verified as soon as it is noted that $\mathscr P_\mu$ and $\mathscr P_e$ are self-adjoint operators, which follows from
 the fact that
     \begin{align}
         \big[\!\left(f_\mu\right)_e\!\big]^{\prime}(t)=\frac{f_e(\overline e)-f_e(\underline e)}{\text{len}(e)},\qquad e\in\mathcal E(\mathcal G),\ t\in(\underline e,\overline e).\label{ea:derivativefmu}
     \end{align}
 \end{proof}

 \begin{lemma}\label{ea:repkernel}
     Let $\mathcal G$ be a graph with Euclidean edges with vertices $\mathcal V$ and edges $\mathcal E$.
     Also let
     \begin{itemize}
         \item $(\mathbb R^{\mathcal V}, \langle \cdot, \cdot\rangle_{L})$ denote the finite dimensional Hilbert space with inner product given by $\langle z, w\rangle_{L}=z^TLw$ as in \eqref{ea:zTLw};
         \item $H_e$ denote the infinite dimensional Hilbert space of absolutely continuous functions $f:[\,\underline e,\,\overline e\,]\to\mathbb R$ such that  $f^\prime \in L^2([\,\underline e,\, \overline e\,])$ with boundary condition $f(\underline e)=f(\overline e) = 0$, and with inner product
     $\langle f, g \rangle_{H_e} := \int_{\underline e}^{\overline e} f^\prime (t) g^\prime (t)\,\mathrm  dt$.
 \end{itemize}
     Then we have the following.
     \begin{enumerate}[{\rm (A)}]
         \item\label{the linear interpolation ismorphic space}
         $(\mathscr P_\mu \mathcal F, \langle \cdot, \cdot\rangle_\mu)$ is a finite dimensional Hilbert space which is isomorphic to $(\mathbb R^{\mathcal V},\langle\cdot,\cdot\rangle_{L})$ and has reproducing kernel $R_\mu$ (defined in \eqref{ea:R_mu in main txt}). Its inner product has simplified form for all  $f,g\in\mathscr P_\mu \mathcal F$:
         \begin{equation}\label{nodeEA Dirichlet form}
             \big\langle f, g \big\rangle_{\mu} = f(u_o)\,g(u_o) \,+\, \sum_{e\in \mathcal E}\frac{\left(f_e(\overline e)-f_e(\underline e)\right)\left(g_e(\overline e)-g_e(\underline e)\right)}{\text{\rm len}(e)}. 
         \end{equation}

         \item\label{residual ismorphic space}
         For each $e\in \mathcal E$,  $(\mathscr P_e \mathcal F, \langle \cdot, \cdot\rangle_{e,r})$ is an infinite dimensional Hilbert space which is isomorphic to $(H_e,\langle\cdot,\cdot\rangle_{H_e})$ and has reproducing kernel $R_e$ (given by \eqref{ea:R_e given in main txt}). Its inner product has simplified form for all $f,g\in\mathscr P_e \mathcal F$:
         \begin{equation}\label{BrownianB Dirichlet form}
             \big\langle f, g \big\rangle_{e,r} = \int_{\underline e}^{\overline e} f_e^{\,\prime} (t) g_e^{\,\prime}(t)\,\mathrm dt. 
         \end{equation}

         \item\label{total ismorphic space}
         $(\mathcal F, \langle \cdot, \cdot \rangle_{\mathcal F})$ is an infinite dimensional Hilbert space which is isomorphic to
         $\mathbb R^{\mathcal V}\otimes\bigotimes_{e\in \mathcal E}H_e$ and has reproducing kernel $R_{\mathcal G}$ (given by \eqref{ea:cov fun of Z_G}). Its inner product 
         has simplified form for all $f,g\in\mathcal F$:
         \begin{equation}\label{Full Dirichlet form}
             \big\langle f, g \big\rangle_{\mathcal F} = \big\langle f_\mu, g_\mu\big\rangle_\mu   \,+\, \sum_{e\in \mathcal E}\, \big\langle f_{e,r}, g_{e,r} \big\rangle_{e,r}.
         \end{equation}
     \end{enumerate}

 \end{lemma}

 \begin{proof}

     \textit{\eqref{the linear interpolation ismorphic space}}:
     There is a bijective correspondence between $z\in \mathbb R^{\mathcal V}$ and $f_\mu \in  \mathscr P_\mu\mathcal F$ which
     simply corresponds  to interpreting $z$ as the values of $f_\mu$ on the vertices of $\mathcal G$.
Then
     \begin{align*}
         f_\mu(u) &= (1-d(u)) z(\underline u) + d(u) z(\overline u),\quad\forall u\in\mathcal G \\
         z(v) &= f_\mu(v),\quad\forall v\in\mathcal V.
     \end{align*}
     The bijection also preserves inner product because if $w\in\mathbb R^{\mathcal V}$ corresponds to $g_\mu\in \mathscr P_\mu\mathcal F$, then
     \begin{align*}
     \langle z,w \rangle_L
     &= z(u_o)w(u_o)
         \,+\! \frac{1}{2}\sum_{u\sim v}\frac{(z(u)-z(v))(w(u)-w(v))}{d_{G,\mathcal G}(u,v)}\qquad \text{(by \eqref{ea:zTLw}}) \\
     &= f_\mu(u_o)\,g_\mu(u_o) \,+\,  \frac{1}{2}\sum_{u\sim v} \frac{\left(f_\mu(u)-f_\mu(v)\right)\left(g_\mu(u)-g_\mu(v)\right)}{\text{\rm len}(e)} \\
     &= \langle f_\mu,g_\mu \rangle_\mu\qquad\text{(by \eqref{ea:derivativefmu}),}
     \end{align*}
     where the above sums are over adjacent $u,v\in\mathcal V$. This establishes \eqref{nodeEA Dirichlet form} and that $(\mathscr P_\mu \mathcal F, \langle \cdot, \cdot\rangle_\mu)$ is isomorphic to $(\mathbb R^{\mathcal V},\langle\cdot,\cdot\rangle_{L})$.
     It then remains to show that the reproducing kernel of $(\mathbb R^{\mathcal V}, \langle \cdot, \cdot\rangle_L)$, namely $L^{-1}$, is in bijective correspondence with $R_\mu$. Indeed, for each $u\in\mathcal G$,
     \begin{align}
         f_\mu(u) &= [\mathscr P_\mu f_\mu](u)\qquad\text{(since $\mathscr P_\mu^2 = \mathscr P_\mu$)} \nonumber\\
         &= (1-d(u)) f_\mu(\underline u) + d(u) f_\mu(\overline u) \nonumber\\
         &= (1-d(u)) z(\underline u) + d(u) z(\overline u) \nonumber\\
         &= (1-d(u)) \big\langle z, L^{-1}(\cdot, \underline u)\big\rangle_L + d(u) \big\langle z, L^{-1}(\cdot, \overline u)\big\rangle_L \nonumber\\
         &=  \big\langle z, \underbrace{(1-d(u)) L^{-1}(\cdot, \underline u) +  d(u)L^{-1}(\cdot, \overline u)}_{:=R_L(\cdot,u)}\big\rangle_L\label{eq:step1 for Rmu}
     \end{align}
     where  the function $R_L(\cdot,u)$ is a member of $ \mathbb R^{\mathcal V}$.
     By \eqref{ea:R_mu in main txt}, we can simply linear interpolate $R_L(\cdot,u)$ to find the corresponding member in $\mathscr P_\mu\mathcal F$ as follows
     \[
        (1-d(\cdot))R_L(\underline \cdot,u) + d(\cdot)R_L(\overline \cdot,u) 
        = R_\mu(\cdot,u).
    \]
    Therefore, by \eqref{eq:step1 for Rmu},
    \begin{align*}\label{eq: step2 for Rmu}
        f_\mu(u)  
        = \big\langle z, R_L(\cdot,u)\big\rangle_L = \big\langle f_\mu, R_\mu(\cdot,u)\big\rangle_\mu
    \end{align*}
    where the second equality follows by the fact that inner products are preserved under the bijective correspondence. This completes the proof of \eqref{the linear interpolation ismorphic space}.

    \textit{\eqref{residual ismorphic space}}:
     Let $e\in\mathcal E$. Note that $H_e$ is equal to the constrained space $\{ f\in \mathscr H_e: f(\overline e) = 0\}$ where $\mathscr H_e := \big\{f\in C([\,\underline e,\, \overline e\,]): f^\prime \in L^2([\,\underline e,\, \overline e\,]),\,  f(\underline e)= 0 \big\}$ corresponds to the {\it Cameron-Martin} Hilbert space (using inner product $\langle \cdot, \cdot\rangle_{H_e}$) with reproducing kernel $(s-\underline e) \wedge (t-\underline e)$. Therefore, by \cite{saitoh1997integral} page 77,  the subspace $(H_e, \langle \cdot, \cdot \rangle_{H_e})$  is also a Hilbert space with reproducing kernel given by
     \begin{equation}\label{ReJM kernel}
         \tilde R_e(s,t) := (s-\underline e) \wedge (t-\underline e) - \frac{(s-\underline e)(t-\underline e)}{\overline e - \underline e}\qquad \underline e<s,t<\overline e.
     \end{equation}
     Clearly,  $f\in H_e$ and $f_{e,r}\in\mathscr P_e\mathcal F$ are in a bijective linear correspondence by the relation $f_{e,r}(u) = f(\varphi_e(u))I_e(u)$. By \eqref{ea:R_e given in main txt},
     \begin{equation}
         R_e(u,v)=\begin{cases}
         \tilde R_e(\varphi_e(u),\varphi_e(v)) & \text{ if $u,v\in e$},\\
         0 & \text{ otherwise.}
     \end{cases}
 \end{equation}
 Finally,
 for $f_{e,r},g_{e,r}\in \mathscr P_e\mathcal F$ with corresponding $f,g\in H_e$\,, we obtain
 $\big\langle f_{e,r},g_{e,r}\big\rangle_{e,r}=\big\langle f,g\big\rangle_{H_e}$,
 and so
 \[\big\langle f_{e,r},R_e(\cdot,v)\big\rangle_{e,r}=\big\langle f,\tilde R_e(\cdot,\varphi_e(v)\big\rangle_{H_e}=f(\varphi_e(v))=f_{e,r}(v).\]
 Thereby (B) is verified.

 \textit{\eqref{total ismorphic space}}: This follows immediately from Lemma \ref{lemma on direct sum} and \eqref{the linear interpolation ismorphic space}-\eqref{residual ismorphic space}.
 \end{proof}

 \begin{proof}[\small{PROOF OF PROPOSITION~\ref{the RKHS associated with Z_G}}]
     Lemma \ref{ea:repkernel}, establishes that $R_{\mathcal G}$ is the reproducing kernel for $(\mathcal F, \langle\cdot, \cdot\rangle_{\mathcal F})$. Since  $R_{\mathcal G}$ is also the covariance function of $Z_{\mathcal G}$ we have that $(\mathcal F, \langle\cdot, \cdot\rangle_{\mathcal F})$ is {\it the} RKHS associated to $Z_{\mathcal G}$ (see \cite{wahba1990}). To prove \eqref{the RKHS inner product associated with Z_G} we use standard Hilbert space arguments to show
     \begin{equation}\label{ea:sup char of var}
     \text{\rm var}(Z_{\mathcal G}(u) - Z_{\mathcal G}(v)) = \sup_{f\in \mathcal F} \left\{(f(u) - f(v) )^2:  \| f \|_{\mathcal F}\leq 1 \right\}
     \end{equation}
     the left hand side being the definition of $d_{R,\mathcal G}(u,v)$.
     Let $u,v\in \mathcal G$ with $u\neq v$, and $f\in\mathcal F $ with $\|f\|_{\mathcal F}\leq 1$.  Use the reproducing property of $R_{\mathcal G}$ along with the Cauchy-Schwartz inequality to obtain
     \begin{align*}
     \left(f(u)-f(v)\right)^2 &= \bigl\langle f,\,R_{\mathcal G}(\cdot,u)\!-\!R_{\mathcal G}(\cdot,v)\bigr\rangle_{\mathcal F}^2 \\
     &\leq \big\|R_{\mathcal G}(\cdot,u)-R_{\mathcal G}(\cdot,v)\big\|^2_{\mathcal F}\\
     &=R_{\mathcal G}(u,u) + R_{\mathcal G}(v,v) - 2R_{\mathcal G}(u,v) \\
     &=\text{\rm var}(Z_{\mathcal G}(u) - Z_{\mathcal G}(v)).
    \end{align*}
    Note that the function $f_o\in\mathcal F$ defined by 
    \[
    f_o(w) = (R_{\mathcal G}(w,u)-R_{\mathcal G}(w,v)) / \|R_{\mathcal G}(\cdot,u)-R_{\mathcal G}(\cdot,v)\|_{\mathcal F}
    \]
     has norm $\|f_o\|_{\mathcal F}= 1$ and satisfies
    \[
    \left(f_o(u)-f_o(v)\right)^2 = \text{\rm var}(Z_{\mathcal G}(u) - Z_{\mathcal G}(v)).
    \]
    This proves \eqref{ea:sup char of var} and hence \eqref{the RKHS inner product associated with Z_G}.

    To show \eqref{dR to R} notice that the definition of inner product on $\mathcal F$, in \eqref{def of inner product}, implies that  $\big\langle f, \textbf{1}\big\rangle_{\mathcal F} = f(u_o)$ where $\textbf{1}$ denotes the member of $\mathcal F$ which has constant value $1$ over all points on $\mathcal G$. Now the reproducing kernel property of $R_{\mathcal G}$ implies $\big\langle f, \textbf{1}\big\rangle_{\mathcal F} = \big\langle f, R_{\mathcal G}(\cdot, u_o)\big\rangle_{\mathcal F}$ for all $f\in \mathcal F$. Therefore $R_{\mathcal G}(\cdot, u_o) =\textbf{1} $ and hence $R_{\mathcal G}(u, u_o) = 1$ for all $u\in\mathcal G$. Finally notice
    \begin{align*}
    d_{R,\mathcal G}(u,u_o) &= R_{\mathcal G}(u,u) + R_{\mathcal G}(u_o,u_o) - 2R_{\mathcal G}(u,u_o) = R_{\mathcal G}(u,u) - 1
    \end{align*}
    which gives
    \begin{align*}
    d_{R,\mathcal G}(u,v) &= R_{\mathcal G}(u,u) + R_{\mathcal G}(v,v) - 2R_{\mathcal G}(u,v) \\
    &= 2 + d_{R,\mathcal G}(u,u_o) + d_{R,\mathcal G}(v,u_o) - 2R_{\mathcal G}(u,v).
    \end{align*}
    This proves \eqref{dR to R} as was to be shown.
 \end{proof}


\subsection{Proofs of Propositions \ref{classical effective resistance}, \ref{splitting and merging edges}, and \ref{relation btwn dG and dr}}\label{a:2}

We start by verifying Proposition~\ref{splitting and merging edges} as it is used to prove Proposition~\ref{classical effective resistance}.

\begin{proof}[\small{PROOF OF PROPOSITION~\ref{splitting and merging edges}
}]

Since the operation of merging two edges at a degree two vertex $v$ is the inverse of splitting the resulting edge at $v$, it will be sufficient to show that $\mathcal G^\prime$ is isometric to $\mathcal G$ under the resistance metric when $\mathcal G^\prime$ is obtained from $\mathcal G$ by splitting edge $e\in\mathcal E(\mathcal G)$ at $u\in e$.

Let $e_1$ and $e_2$ denote the partial edges formed by splitting $e\in \mathcal E(\mathcal G)$ at $u\in e$ such that $\underline e_1=\underline e$. Since the sets $\mathcal G$ and $\mathcal G^\prime$ are identical, their corresponding spaces of functions $\mathcal F$ and $\mathcal F^\prime$ as given in Definition \ref{ea:diriclet form def} are identical, however, $\mathcal G$ and $\mathcal G^\prime$ induce different inner products on $\mathcal F$, denoted $\langle \cdot, \cdot \rangle_{\mathcal F}$ and $\langle \cdot, \cdot \rangle_{\mathcal F, \text{split}}$, respectively. By Proposition~\ref{the RKHS associated with Z_G}, $d_{R,\mathcal G}$ and $d_{R,\mathcal G^\prime}$ are completely determined by their inner products $\langle f, g \rangle_{\mathcal F, \text{split}}$ and $\langle f, g \rangle_{\mathcal F}$. Therefore, we may suppose that both $\mathcal G$ and $\mathcal G^\prime$ use the same origin $u_o$ in their respective inner products. Then,
 for any $f, g \in \mathcal F$, the difference between the two inner products is 
       \begin{align*}
           &\langle f, g \rangle_{\mathcal F} - \langle f, g \rangle_{\mathcal F, \text{split}}\\
           =\,&
         \int_{\underline e}^{\overline e}\! f_e^{\,\prime} (t) g_e^{\,\prime}(t) \,\mathrm dt
           -\!\int_{\underline e_1}^{\overline e_1}\!
           f_{e_1}^{\,\prime} (t) g_{e_1}^{\,\prime}(t) \,\mathrm dt
           \,-\! \int_{\underline e_2}^{\overline e_2}\! f_{e_2}^{\,\prime} (t) g_{e_2}^{\,\prime}(t) \,\mathrm dt
       \end{align*}
       since the splitting operation on $e\in\mathcal E(\mathcal G)$ at $u\in e$ only affects the term corresponding to $e$ in 
       (\ref{def of inner product}). By Proposition~\ref{the RKHS associated with Z_G}, to show $d_{R,\mathcal G} = d_{R,\mathcal G^\prime}$, it will be sufficient to show  $\langle f, g \rangle_{\mathcal F}=\langle f, g \rangle_{\mathcal F, \text{split}}$ for all  $f, g\in\mathcal F$.

       For any  $f\in \mathcal F$ and $t\in[\underline e,\overline e]$, define
       \[
       f_1(t) =  f_{e}(t)\,I_{[\underline e_1, \overline e_1)}(t),\qquad f_2(t) =  f_{e}(t)\,I_{[\underline e_2, \overline e_2]}(t).
       \]
       Both  $f_1$ and $f_2$ are almost everywhere differentiable and satisfy $f_1^{\,\prime}, f_2^{\,\prime}\in L^2([\underline e, \overline e])$. Moreover, for any $f,g\in \mathcal F$, the fact that $f^{\,\prime}_1(t)g^{\,\prime}_2(t) \overset{a.e.}= 0$ and $f^{\,\prime}_2(t)g^{\,\prime}_1(t)\overset{a.e.}= 0$ implies that
       \begin{align*}
       \int_{\underline e}^{\overline e}\! f_e^{\,\prime} (t) g_e^{\,\prime}(t) \,\mathrm dt
       &= \int_{\underline e}^{\overline e}\! \big[f_1^{\,\prime}(t)+ f_2^{\,\prime}(t)\big]  \big[g_1^{\,\prime}(t)+ g_2^{\,\prime}(t)\big]  \,\mathrm dt \\
    &= \int_{\underline e_1}^{\overline e_1}\! f_1^{\,\prime} (t) g_1^{\,\prime}(t) \,\mathrm dt + \int_{\underline e_2}^{\overline e_2}\! f_2^{\,\prime} (t) g_2^{\,\prime}(t) \,\mathrm dt.
    \end{align*}
    Note that for Lebesgue almost all numbers $t$,
    $f^{\,\prime}_{e_1}(t) = f^{\,\prime}_1(t)$ if $t\in [\underline e_1, \overline e_1]$ and $f^{\,\prime}_{e_2}(t) = f^{\,\prime}_2(t)$ if $t\in [\underline e_2, \overline e_2]$ (and similarly for $g_{e_1}$, $g_{e_2}$). Therefore,
    \begin{align*}
       \int_{\underline e}^{\overline e}\! f_e^{\,\prime} (t) g_e^{\,\prime}(t) \,\mathrm dt
       &= \int_{\underline e_1}^{\overline e_1}\!
           f_{e_1}^{\,\prime} (t) g_{e_1}^{\,\prime}(t) \,\mathrm dt
           \,+\! \int_{\underline e_2}^{\overline e_2}\! f_{e_2}^{\,\prime} (t) g_{e_2}^{\,\prime}(t) \,\mathrm dt
       \end{align*}
       which implies $\langle f, g \rangle_{\mathcal F, \text{split}} = \langle f, g \rangle_{\mathcal F}$ as was to be shown.
\end{proof}

\begin{proof}[\small{PROOF OF PROPOSITION~\ref{classical effective resistance}
}]

    In the literature on resistance networks and metrics \citep[see e.g.][]{Kigami:03,JorgensenPearse:10}, given a conductance function $c$ (i.e., a symmetric function associated to all pairs of adjacent vertices),
    the (effective) resistance distance between $u,v\in \mathcal V(\mathcal G)$
     is defined by
    \begin{equation}\label{ea:usualDEF}
        d_{\text{eff}}(u,v) = \sup_{z\in\mathbb R^{\mathcal V(\mathcal G)}}\Big\{(z(u)-z(v))^2:\text{
        $\frac{1}{2}\sum_{u_1\sim u_2}c(u_1,u_2)(z(u_1)-z(u_2))^2\le 1$}\Big\}
    \end{equation}
     (this is one of several equivalent definitions, cf.\ Theorem 2.3 in \cite{JorgensenPearse:10}).
    To relate $d_{\text{eff}}$ and $d_{R, \mathcal G}$, recall \eqref{the RKHS inner product associated with Z_G} and that we have defined $c$ by \eqref{conductance1}.
    Also, by Lemma \ref{ea:repkernel}, each $f\in\mathcal F$ has an orthogonal decomposition $f = f_\mu + \sum_{e\in \mathcal E(\mathcal G)} f_{e,r}$ where
    \[
        \|f\|^2_{\mathcal F} = \|f_\mu\|_\mu^2   \,+\, \sum_{e\in \mathcal E(\mathcal G)}\, \|f_{e,r}\|^2_{e,r}
    \]
    and $f_{e,r}(u) = f_{e,r}(u) = 0 $ for all $u,v\in\mathcal V(\mathcal G)$.
    Therefore, if $u,v\in\mathcal V(\mathcal G)$,
    the term $(f(u) - f(v))^2$ in \eqref{the RKHS inner product associated with Z_G} simplifies to $(f_\mu(u) - f_\mu(v))^2$ and hence the supremum can be taken over $f\in\mathcal F$ such that $\|f_{e,r}\|^2_{e,r}=0$ for all $e\in \mathcal E(\mathcal G)$. By Lemma \ref{ea:repkernel} \eqref{the linear interpolation ismorphic space}, when $u,v\in\mathcal V(\mathcal G)$, $d_{R,\mathcal G}(u,v)$ is equal to
    \begin{equation*}
        \sup_{f\in \mathscr P_\mu \mathcal F} \Big\{(f_\mu(u) - f_\mu(v) )^2: \text{ $f_\mu(u_o)^2 \,+\, \sum_{e\in \mathcal E(\mathcal G)}\frac{\big([f_\mu]_e(\overline e)-[f_\mu]_e(\underline e)\big)^2}{\text{\rm len}(e)} \leq 1 $}\Big\}.\label{ea:our DEF}
    \end{equation*}
    Since the constant functions are all members of $\mathscr P_\mu \mathcal F$, we can subtract $f_\mu(u_o)$ from each $f_\mu\in \mathscr P_\mu \mathcal F$ and easily see that the supremum above 
    can be taken over all $f\in \mathscr P_\mu \mathcal F$ which satisfy $f_\mu(u_o) = 0$. It is now easily seen that
    \[ d_{\text{eff}}(u,v) = d_{R,\mathcal G}(u,v),\qquad  u,v\in\mathcal V(\mathcal G). \]
    This also establishes that $d_{R,\mathcal G}$ does not depend on the choice of origin and that $d_{R,\mathcal G}$ is a metric on $\mathcal V(\mathcal G)$ \citep[see e.g.][Lemma 2.6]{JorgensenPearse:10},
    and hence by the splitting operation on edges (Proposition~\ref{splitting and merging edges}) $d_{R,\mathcal G}$ is a metric on $\mathcal G$ as well. 
\end{proof}

\begin{proof}[\small{PROOF OF PROPOSITION~\ref{relation btwn dG and dr}}]

Proposition~\ref{classical effective resistance}
and the theory of electrical networks imply 
\begin{equation}\label{ea:dRdG}
    d_{R, \mathcal G}(u,v)\le d_{G, \mathcal G}(u,v),\qquad u,v\in \mathcal V(\mathcal G),
\end{equation}
with equality if and only if $\mathcal G$ is a tree graph \citep[see e.g.][Lemma 4.3]{JorgensenPearse:10}. The fact that $d_{R,\mathcal G}$ and  $d_{G,\mathcal G}$ are invariant to splitting edges (by Proposition~\ref{splitting and merging edges}) implies that \eqref{ea:dRdG} extends to any additional finite collection of edge points. Thereby, \eqref{eq: dr <= dg} of
Proposition~\ref{relation btwn dG and dr} is verified, where the \textit{if and only if} follows since the tree property of $\mathcal G$ is also invariant to edge splitting.

To show \eqref{eq: form of d_R when G is a cycle} suppose $\mathcal G$ is a Euclidean cycle with circumference $\omega$. Let $u,v\in \mathcal G$ be arbitrary and $s \in \mathcal G$ be the polar opposite of the midpoint of the geodesic path connecting $u$ to $v$. By a sequence of edge splits and merges we may construct a new graph $\mathcal G^\prime$, equaling $\mathcal G$ as a point set, but with vertices $\{u, v, s\}$ and edges connecting $u\sim v$, $v\sim s$ and $u\sim s$ with corresponding edge lengths $d_{G,\mathcal G}(u,v)$, $d_{G,\mathcal G}(v,s)$ and $d_{G,\mathcal G}(u,s)$, respectively.  Notice that the $L$ matrix for $\mathcal G^\prime$, constructed via \eqref{def of L in main txt}, has a particularly simple inverse given by
\begin{align*}
L^{-1} &=
\begin{pmatrix}
            c_1 \!+\! c_2   & -c_1        & -c_2 \\
            -c_1      & c_1 \!+\! c_3   & -c_3  \\
            -c_2      & -c_3       &  c_2 \!+\! c_3 \!+\! 1
    \end{pmatrix}^{\!-1} = 
      \frac{1}{b}\begin{pmatrix}
            {b \!+\! c_1\!+\!c_3}   & {b \!+\! c_1}       & b \\
            {b \!+\! c_1}      & {b \!+\! c_1\!+\!c_2}    & b  \\
            b      & b       &  b
        \end{pmatrix}   
\end{align*}
where $c_1:= 1/d_{G,\mathcal G}(u,v)$, $c_2:= 1/d_{G,\mathcal G}(u,s)$, and $c_3:=1/d_{G,\mathcal G}(v,s)$ are the edge conductances of $\mathcal G^\prime$ and 
 $b := c_1 c_2+c_1 c_3+c_2c_3$. Now, since $u,v\in\mathcal V(\mathcal G^\prime)$,
 \begin{align*}
 d_{R,\mathcal G^\prime}(u,v) &= L^{-1}(u,u) + L^{-1}(v,v) - 2L^{-1}(u,v) \\
 &= \frac{b+c_1+c_3}{b} + \frac{b+c_1+c_2}{b} - 2\frac{b+c_1}{b} \\
 &= d_{G,\mathcal G}(u,v) - \frac{d_{G,\mathcal G}(u,v)^2}{\omega}
 \end{align*}
 which, along with the fact that $d_{R,\mathcal G}(u,v) =  d_{R,\mathcal G^\prime}(u,v)$ by Proposition~\ref{splitting and merging edges},  completes the proof of \eqref{eq: form of d_R when G is a cycle}.
\end{proof}

\bibliography{refs}

\end{document}